\documentclass[11pt, leqno]{article}
\usepackage{amsmath,amssymb,amsbsy,amsfonts,amsthm,latexsym,
            amsopn,amstext,amsxtra,euscript,amscd,amsthm}
\usepackage{mathrsfs}
\usepackage{enumerate}
\newtheorem{lem}{Lemma}
\newtheorem{lemma}[lem]{Lemma}

\newtheorem{thm}{Theorem}
\newtheorem{theorem}[thm]{Theorem}

\def\\{\cr}
\def\({\left(}
\def\){\right)}
\def\[{\left[}
\def\]{\right]}
\def\<{\langle}
\def\>{\rangle}

\def\al{{\alpha}}
\def\be{{\beta}}
\def\ep{{\varepsilon}}

\def\cD{{\mathcal D}}

\def\cF{{\mathcal F}}

\def\Z{{\mathbb Z}}

\def\C{{\mathbb C}}

\begin{document}

\title{On members of Lucas sequences which are either products of factorials or product of 
middle binomial coefficients and Catalan numbers}

\author{{\sc Shanta~Laishram}\\
{Stat-Math Unit, Indian Statistical Institute}\\
{7, S. J. S. Sansanwal Marg, New Delhi, 110016, India}\\
{shanta@isid.ac.in}}
\date{} 
\pagenumbering{arabic}

\maketitle

\begin{abstract}
Let $\{U_n\}_{n\ge 0}$ be a Lucas sequence. Then the equation 
$$|U_n|=m_1!m_2!\cdots m_k!$$ 
with $1<m_1\le m_2\le \cdots\le m_k$ implies $n\in \{1,2, 3, 4, 6, 8, 12\}$.  Further the equation 
$$|U_n|=D_{m_1}D_{m_2}\cdots D_{m_k}, \qquad D_{m_i}\in \{B_{m_i}, C_{m_i}\}$$ 
with $1<m_1\le m_2\le \cdots\le m_k$ implies 
$n\in \{1,2, 3, 4, 6, 8, 12, 16\}$. Here $B_m$ is the middle binomial coefficient 
$\binom{2m}{m}$ and $C_m$ is the Catalan number $\frac{1}{m+1}\binom{2m}{m}$. 
\end{abstract}

\section{Introduction}

Let $r,~s$ be coprime nonzero integers with $r^2+4s\ne 0$. Let $\alpha,~\beta$ be the roots of the quadratic equation 
$x^2-rx-s=0$ 
and assume without loss of generality that $|\al|\geq \be|$. We assume further that 
$\alpha/\beta$ is not a root of $1$. The Lucas sequences $\{U_n\}_{n\ge 0}$ and $\{V_n\}_{n\ge 0}$ of parameters $(r,s)$ are given by 
$$
U_n=\frac{\alpha^n-\beta^n}{\alpha-\beta}\qquad {\text{\rm and}}\qquad V_n=\alpha^n+\beta^n\qquad {\text{\rm for~all}}\qquad n\ge 0.
$$
Alternatively, they can be defined recursively as $U_0=0,~U_1=1,~V_0=2,~V_1=r$ and both recurrences
$$
U_{n+2}=rU_{n+1}+sU_n\quad {\text{\rm and}}\quad V_{n+2}=rV_{n+1}+sV_n\qquad  {\text{\rm hold for all}}\quad n\ge 0.
$$ 
Let 
$$
{\mathcal PF}:=\{\pm 1\}\cup \{\pm \prod_{j=1}^k m_j!: 1<m_1\le m_2\le\cdots \le m_k~{\text{\rm and}}~k\ge 1\}
$$
be the set of integers whose absolute values are the product of factorials $>1$.  Members of ${\mathcal PF}$ are 
sometimes called {\it Jordan-Polya} numbers and several recent papers investigate their arithmetic properties. 
It was shown in \cite{LLS-F} that $U_n\in {\mathcal PF}$ implies 
$n<62000$. We reduce this bound to $n\in \{1, 2, 3, 4, 6, 8, 12\}$ in the following theorem.

\begin{theorem}\label{thmF}
The equation  
\begin{align}\label{eqn1}
U_n=\pm m_1!m_2!\cdots m_k! \quad { where} \ k\geq 1 \quad { and} \quad 1<m_1\leq m_2\leq \cdots \leq m_k
\end{align}
implies $n\in \{1, 2, 3, 4, 6, 8, 12\}$. 

The equation
\begin{align}\label{eqn2}
V_n=\pm m_1!m_2!\cdots m_k, \qquad 1<m_1\leq m_2\leq \cdots \leq m_k.
\end{align}
implies $n=2$ or $n$ is an odd prime $\leq 787$. Further $n\in \{1, 2, 3\}$ when $\alpha, \beta $ are real. 
\end{theorem}

It was proved in \cite{LLS-F} that for real $\alpha, \beta $,  the equation \eqref{eqn1} implies 
$n\in \{1,2, 3, 4, 6,12\}$ and the equation \eqref{eqn2} implies $n\in \{1, 2, 3\}$.  We believe that 
that  the equation \eqref{eqn1} has no solution for $n=8$ which we are not able to prove. 

Let 
$$
B_m:=\binom{2m}{m}\quad {\text{\rm and}}\quad C_m:=\frac{1}{m+1}\binom{2m}{m}\quad {\text{\rm for}}\quad m\ge 0,
$$
be the middle binomial coefficient and Catalan number, respectively. For each $m$, we write $D_m$ 
for one of the numbers $B_m,C_m$.  Let 
$$
{\mathcal PBC}:=\{\pm \prod_{j=1}^k D_{m_j}: D_m\in \{B_m,C_m\}, ~k\ge 1,~ 1\le m_1\le m_2\le\cdots \le m_k\}
$$
be the set of integers which are products of middle binomial coefficients and Catalan numbers. It was proved 
in \cite{LLS-C} that if $U_n\in {\mathcal PBC}$, then $n<6500$ if $n$ is odd and $n\leq 720$ if $n$ is even. 
We reduce this bound to $n\in \{1, 2, 3, 4, 6, 8, 12\}$ in the following theorem. 

\begin{theorem}\label{thmC}
For each $m$, let $D_m\in \{B_m,C_m\}$. The equation  
\begin{align}\label{eqn3}
U_n=\pm D_{m_1} D_{m_2}\cdots D_{m_k},\quad {\text{where}}\quad k\ge 1 \quad {\text{and}}\quad 1\le m_1\le \cdots\le m_k,
\end{align}
implies $n\in \{1, 2, 3, 4, 6, 8, 12, 16\}$. 

The equation
\begin{align}\label{eqn4}
V_n=\pm D_{m_1} D_{m_2}\cdots D_{m_k},\quad {\text{where}}\quad k\ge 1 \quad {\text{and}}\quad 1\le m_1\le \cdots\le m_k,
\end{align}
implies $n\leq 787$ is prime or $n=2p$ for a prime $p\leq 397$. Further, $n\in \{1, 2, 3, 6\}$
when $\alpha, \beta $ are real.    
\end{theorem}

It was proved in \cite{LLS-C} that for real $\alpha, \beta $,  the equation \eqref{eqn3} implies 
$n\in \{1,2, 3, 4, 6, 8, 12\}$ and the equation \eqref{eqn4} implies $n\in \{1, 2, 3, 6\}$.  We believe that 
that  the equation \eqref{eqn3} has no solution for $n=16$ which we are not able able to prove. 

It was conjectured in \cite{LLS-F} that there are only finitely many solutions of \eqref{eqn1} with 
$n\in \{6, 12\}$ regardless or whether 
$\alpha,\beta$ are real or complex conjugates unconditionally which is still open. In fact it was also 
conjectured that the solutions of equation \eqref{eqn1} are given by 
$(n; r, s)=(12; \pm 1, 1), U_n=\pm 144$ and 
\begin{align*}
(n; r, s)\in \begin{split}
\{&(5; \pm 1, -2), (5; 1, -3), (5, \pm 12, -55), \\
&(5; 12, -377), (7; \pm 1, -5),  (13; \pm 1, -2)\} \quad {\rm where} \ U_n=\pm 1
\end{split}
\end{align*} 
and $n=6$
\begin{align*}
(r, s)\in \begin{split}
\{&(\pm 1,  5),  (\pm 1, -7), (\pm 1,  53), (\pm 3, -1), (\pm 3, -5), (\pm 3,  7), (\pm 3, -7),\\
& (\pm 3, -11),  (\pm 3, -19),  (\pm 3,  61),  (\pm 3, -73), (\pm 5, -19),  (\pm 5,  29),  \\
&(\pm 6, -97),   (\pm 9, -11),   (\pm 9, -17), (\pm 9, -91), (\pm 9, -97)\}.
\end{split}
\end{align*} 
It was also conjectured in \cite{LLS-F} that the only solutions $(r, s, n)$ of \eqref{eqn2} with 
$n\geq 4$ are given by $(r, s, n)\in \{(\pm 1, 1, 6), (\pm 2, -1, 4)\}$. We refer to \cite{LLS-F} for more 
details. We also refer to \cite{LLS-C} for results and problems on the equations \eqref{eqn3} and 
\eqref{eqn4}.

We give the proof of Theorems \ref{thmF} and \ref{thmC} in Section 3.  Throughout the proof, 
we use $\omega(n), P(n), \mu(n)$ and $\varphi(n)$ with the regular meaning as 
being the number of distinct prime factors of $n$, the largest prime factor of $n$, the 
M\"obius function of $n$ and the Euler phi function of $n$, respectively.   All the computations in this 
manuscript were carried out in SageMath.  

\section{Preliminaries}

Let $n_0$ be a positive integer. For an integer $\ell$, define
\begin{equation}\label{Mn}
M_{n_0}(\ell):=\log \left( \prod_{\substack{p^{\nu_p}\|\ell \\ 
p\equiv \pm 1\pmod {n_0}}} p^{\nu_p}\right)=\sum_{\substack{p^{\nu_p}\|\ell \\ 
p\equiv \pm 1\pmod {n_0}}} \nu_p\log p.
\end{equation}
We prove a number of results to estimate lower and upper bounds for $M_{n_0}(U_n)$ and 
$M_{n_0}(V_n)$ for some divisors $n_0$ of $n$.  

To recall the terminology, we take coprime nonzero integers $r, s$ with $r^2+4s\neq 0$ and  let 
$\alpha$ and $\beta$ be the roots of the equation $\lambda^2-r\lambda-s=0$. For $n\geq 0$, we have 
\begin{align*}
U_n=\frac{\alpha^n-\beta^n}{\alpha-\beta} \quad {\rm and} \quad 
V_n= \alpha^n+\beta^n.
\end{align*}
We suppose that $\alpha/\beta$ is not a root of unity.  We 
assume without loss of generality that $|\alpha|\ge |\beta|$. Further, we may 
replace $(\alpha,\beta)$ by $(-\alpha,-\beta)$ if needed. This replacement changes the pair 
$(r,s)$ to $(-r,s)$, while $|U_n|$ and $|V_n|$ are not affected 
and hence the values of $M_{n_0}(|U_n|)$ and $M_{n_0}(|V_n|)$ for any divisor $n_0$ of $n$. 
Thus, we may assume that $r>0$. When $\alpha,\beta$ are real, these conventions imply that $\alpha$ is positive 
so $\alpha>|\beta|$.  Further, in this case $U_n>0$ and $V_n>0$ for all $n\ge 1$. 

We begin by proving a lower bound for $M_{n_0}(U_n)$ and $M_{n_0}(V_n)$ for some divisors $n_0$ of $n$. 
Throughout the paper, we use $x:=\beta/\alpha$. 

\begin{lemma}\label{primn_0}
Let $n$ be a positive integer and $p<p_1$ be distinct primes and $t\geq 0, h>0, h_1>0$ be integers. 
Let $n_0\in \{p^h, p^hp^{h_1}_1\}$, $n_0>4, n_0\notin \{6, 12\}$ be such that $n_0p^t\mid n$.  
 Then 
\begin{align}\label{n_0U}
\begin{split}
 M_{n_0}(U_n)\geq 
\begin{cases}
\log \frac{|\al^n-\be^n|}{|\al^{\frac{n}{p^{t+1}}}-\be^{\frac{n}{p^{t+1}}}|}-\log p^{t+1}, &  n_0=p^{h};\\
\frac{|\al^n-\be^n|}{|\al^{\frac{n}{p^{t+1}}}-\be^{\frac{n}{p^{t+1}}}|}
\frac{|\al^{\frac{n}{p_1p^{t+1}}}-\be^{\frac{n}{p_1p^{t+1}}}|}{
|\al^{n/p_1}-\be^{n/p_1}}|, &  n_0=p^hp^{h_1}_1,\\
-\log P\left(\frac{n_0}{\gcd(n_0, 3)}\right)^{t+1} &
\end{cases} 
\end{split}
\end{align}
and for $n_0=p^h, p>2$, 
\begin{align}\label{n_0V}
\begin{split}
M_{n_0}(V_n) \geq \log \frac{|\al^n+\be^n|}{|\al^{\frac{n}{p^{t+1}}}+\be^{\frac{n}{p^{t+1}}}|}-\log p^{t+1}. 
\end{split}
\end{align}
\end{lemma}

\begin{proof}
Let $n_0$ be the divisor of $n$ given in the statement of the lemma. Let  $m=n_0p^t$. Write
$$U_n=\frac{\al^n-\be^n}{\al-\be}=\left(\frac{(\al^{n/m})^{m}-(\be^{n/m})^{m}}
{\al^{n/m}-\be^{n/m}}\right)\left( \frac{\al^{n/m}-\be^{n/m}}{\al-\be}\right).
$$
Let  $\al_1:=\al^{n/m}$ and $\be_1:=\be^{n/m}$ and put 
$$
U^{1}_{\ell}=\frac{\al^{\ell}_1-\be^{\ell}_1}{\al_1-\be_1}  \quad {\rm and} \quad 
V^{1}_{\ell}=\al^{\ell}_1+\be^{\ell}_1 \qquad {\rm for} \quad \ell\geq 1.
$$ 
Then $\{U^{1}_{\ell}\}_{\ell\geq 0}$ and$\{V^{1}_{\ell}\}_{\ell\geq 0}$  are the Lucas sequences 
with parameters $(r_1, s_1)$, where $(r_1, s_1)=(\al_1+\be_1, -\al_1\be_1)=(V_{n/m}, (-1)^{n/m-1}s^{n/m})$.  Further, we have 
 $U_n=U^1_{m}U_{n/m}$ and $V_n=V^1_{m}$ implying 
 $$M_{n_0}(U_n)\geq M_{n_0}(U^1_{m}) \qquad {\rm and} \qquad  M_{n_0}(V_n)\geq M_{n_0}(V^1_{m}).$$ 
Observe that $U^1_{m}=U^1_{n_0p^t}$ is divisible by each 
$U^1_{n_0p^i}, 0\leq i\leq t$. Recall that a prime $q\mid U^1_{\ell}$ is a 
primitive divisor of $U^1_{\ell}$ if $q\nmid U^1_{\ell'}$ for $\ell'<\ell$ and $q\nmid r^2_1+4s_1$. 
Also the primitive divisors of $U^1_{\ell}$ are all congruent to one of $ \pm 1$ modulo $\ell$. Hence, 
the primitive divisors of  $U^1_{n_0p^i}$ for $0\leq i\leq t$ are all congruent to one of 
$\pm 1$ modulo $n_0$. We now look at the primitive part of $U^1_{\ell}$. This is the part of 
$U^1_{\ell}$  built up only with powers of primitive prime divisors of $U^{1}_{\ell}$.  Thus, 
the primitive parts of $U^1_{n_0p^i}$ for $0\leq i\leq t$ divide $U^1_{m}$. Hence, 
$$M_{n_0}(U_n)\geq M_{n_0}(U^1_{m})\geq M_{n_0}\left(\prod^t_{i=0}U^1_{n_0p^i}\right).$$

For a positive integer $\ell$, let 
\begin{align*}
\Phi_\ell(\alpha_1, \beta_1):=\prod_{\substack{1\le k\le \ell \\ (k, \ell)=1}} (\alpha_1-e^{2\pi i k/\ell} \beta_1)
\end{align*}
be the specialisation of the homogenization $\Phi_\ell(X,Y)$ of the $\ell$-th cyclotomic polynomial $\Phi_\ell(X)$ in 
the pair $(\alpha_1, \beta_1)$. Further, it is well-known (see, for example,  \cite[Theorem 2.4]{BHV}), 
that for $\ell>4, \ell\notin \{6, 12\}$, 
$$
\prod_{\substack{p^{\nu_p}\| U^1_{\ell}\\ p~{\rm primitive}}} p^{\nu_p}=\frac{\Phi_{\ell}(\alpha_1,\beta_1)}{P'(\ell)},
$$
where $P'(\ell):=P(\frac{\ell}{\gcd(\ell, 3)})$. Since primitive 
divisors of $U^1_{\ell}$ are congruent to one of $\pm 1$ modulo $\ell$, we obtain by taking 
$\ell=n_0p^i$ for $0\leq i\leq t$ that 
\begin{align}\label{MUn}
\begin{split}
M_{n_0}(U_n)\geq M_{n_0}\left(\prod^t_{i=0}U^1_{n_0p^i}\right)\geq 
\left(\prod^t_{i=0}|\Phi_{n_0p^i}(\alpha_1, \beta_1)|\right)(P'(n_0))^{-t-1}.
\end{split}
\end{align}
Also from the fact that $V_n=V^1_{n_0p^t}$ is divisible by each $V_{n_0p^i}, 0\leq i\leq t$ (here $n_0, p$ are both odd) 
and the primitive part of $V_{n_0p^i}$ is exactly the primitive part of $U^1_{2n_0p^i}$, we obtain  similarly 
\begin{align}\label{MVn}
M_{n_0}(V_n)\geq 
M_{2n_0}\left(\prod^t_{i=0}U^1_{2n_0p^i}\right)\geq \left(\prod^t_{i=0}|\Phi_{2n_0p^i}(\alpha_1, \beta_1)|\right)
(P'(n_0))^{-t-1}.
\end{align}
Therefore, it remains to estimate the right--hand sides of inequalities \eqref{MUn} and \eqref{MVn}.  

It is well-known that for a positive integer $\ell$, 
\begin{align*}
\Phi_\ell(\alpha_1, \beta_1) = \prod_{d\mid \ell} (\alpha^{\frac{\ell}{d}}_1-\beta^{\frac{\ell}{d}}_1)^{\mu(d)}.
\end{align*}
Hence, we have, by using $\al^{n_0p^t}_1=\al^n$,
\begin{align}\label{MU1}
\begin{split}
\prod^t_{i=0}\Phi_{n_0p^i}(\alpha_1, \beta_1)=&\prod^t_{i=0}\frac{\al^{p^{h+i}}_1-\be^{p^{h+i}}_1}
{\al^{p^{h+i-1}}_1-\be^{p^{h+i-1}}_1}
=\frac{\al^{p^{h+t}}_1-\be^{p^{h+t}}_1}{\al^{p^{h-1}}_1-\be^{p^{h-1}}_1}\\
=&\frac{\al^n-\be^n}{\al^{n/p^{t+1}}-\be^{n/p^{t+1}}}, \quad n_0=p^h;
\end{split}
\end{align}
and 
\begin{align}\label{MU2}
\begin{split}
\prod^t_{i=0}\Phi_{n_0p^i}(\alpha_1,\beta_1)=&\prod^t_{i=0}\frac{(\al^{p^{h+i}p^{h_1}_1}_1-\be^{p^{h+i}p^{h_1}_1}_1)
(\al^{p^{h+i-1}p^{h_1-1}_1}_1-\be^{p^{h+i-1}p^{h_1-1}_1}_1)}
{(\al^{p^{h+i-1}p^{h_1}_1}_1-\be^{p^{h+i-1}p^{h_1}_1}_1)(\al^{p^{h+i}p^{h_1-1}_1}_1-\be^{p^{h+i}p^{h_1-1}_1}_1)}\\
=&\frac{(\al^{p^{h+t}p^{h_1}_1}_1-\be^{p^{h+t}p^{h_1}_1}_1)(\al^{p^{h-1}p^{h_1-1}_1}_1-\be^{p^{h-1}p^{h_1-1}_1}_1)}{
(\al^{p^{h-1}p^{h_1}_1}_1-\be^{p^{h-1}p^{h_1}_1}_1)(\al^{p^{h+t}p^{h_1-1}_1}_1-\be^{p^{h+t}p^{h_1-1}_1}_1)} \\
=&\left(\frac{\al^n-\be^n}{\al^{\frac{n}{p^{t+1}}}-\be^{\frac{n}{p^{t+1}}}}\right)\left(
\frac{\al^{\frac{n}{p_1p^{t+1}}}-\be^{\frac{n}{p_1p^{t+1}}}}{\al^{n/p_1}-\be^{n/p_1}}\right),
 \quad  n_0=p^hp^{h_1}_1. 
\end{split}
\end{align}
Also,
\begin{align}\label{MV1}
\prod^t_{i=0}\Phi_{2n_0p^i}(\alpha_1,\beta_1)
=\frac{\al^{p^{h+t}}_1+\be^{p^{h+t}}_1}{\al^{p^{h-1}}_1+\be^{p^{h-1}}_1}
=\frac{\al^n+\be^n}{\al^{\frac{n}{p^{t+1}}}+\be^{\frac{n}{p^{t+1}}}},\quad n_0=p^h.
\end{align}

From $|\al|\geq |\be|$, we have $|x|\leq 1$. Taking out the powers of $\al$ in 
\eqref{MU1}--\eqref{MV1} and further using in \eqref{MU2} the inequality
\begin{align*}
\left|\frac{1-y}{1-y^{p^{t+1}}}\right|\geq \frac{1}{p^{t+1}} \quad {\text{\rm valid~for~all}}\quad p, \quad {\text{\rm {\rm where}}} \quad 
y:=x^{\frac{n}{p_1p^{t+1}}}\quad {\text{\rm has}}\quad |y|\leq 1, 
\end{align*}
we get the assertions \eqref{n_0U} and \eqref{n_0V} from \eqref{MUn} and \eqref{MVn}, respectively.  
\end{proof}

\begin{lemma}\label{x3/x}
Let $x\in \C$ with $|x|=1$. Then 
$$\left|\frac{1\pm x^3}{1\pm x}\right|=|1\pm x+x^2|\geq \frac{\sqrt{17}-3}{2}>0.56.$$
\end{lemma}

\begin{proof}
Put $\ep=\frac{\sqrt{17}-3}{2}$. First we prove that $|1+x+x^2|\geq \ep$. Write $x=e^{i\theta}$. 
We have 
$$|1+x|^2=(1+\cos \theta)^2+\sin^2 \theta=2(1+\cos \theta)=(2\cos \frac{\theta}{2})^2$$
and 
$$|1+x^2|^2=(1+\cos 2\theta)^2+\sin^2 2\theta=2(1+\cos 2\theta)=(2\cos \theta)^2.$$ 
Assume that $|1+x+x^2|<\ep$. Then 
$$\ep>|1+x+x^2|\geq ||1+x|-|x^2||\geq |2\cos \frac{\theta}{2}-1|$$
implying 
$$1-\ep<|2\cos \frac{\theta}{2}|<1+\ep.$$
Again
\begin{align*}
\ep>|1+x+x^2|&\geq ||1+x^2|-|x||\geq |2\cos \theta-1|=|2(2\cos^2 \frac{\theta}{2}-1)-1|\\
&\geq |3-(2\cos \frac{\theta}{2})^2|\geq 3-(1+\ep)^2=\ep.
\end{align*}
This is a contradiction. Hence $|1+x+x^2|\geq \ep$. Putting $y=-x$, we have $|y|=|-x|=1$ and 
hence $|1-x+x^2|=|1+y+y^2|\geq \ep$. 
\end{proof}

\begin{lemma}\label{23|n}
Let $\al, \be$ be complex conjugates. Assume that $U_n$ and $V_n$ has primitive 
divisors. Let $Q_n$ be the least prime congruent to $\pm 1$ modulo $n$. 
Let $p|n$ be prime with $p$ odd when we consider $\al^n+\be^n$. Then 
\begin{align}\label{p|n<}
|\al^n\pm \be^n|\geq \begin{cases}
|\al^{\frac{n}{p}}\pm \be^{\frac{n}{p}}||\al|^{\frac{n(p-1)}{p}} \quad &{\rm if} 
\quad |\al^{\frac{n}{p}}\pm \be^{\frac{n}{p}}|\leq \frac{2|\al|^{\frac{n}{p}}}{p}\\
\frac{2Q_n|\al|^{\frac{n}{p}}}{p} \quad &{\rm if} \quad |\al^{\frac{n}{p}}\pm \be^{\frac{n}{p}}|>
\frac{2|\al|^{\frac{n}{p}}}{p}. 
\end{cases}
\end{align}
Let $p\leq q$ be primes, $pq|n$ and further  $p, q$  are odd when we consider 
$\al^n+\be^n$.  Put $q_0=p^2, q$ according as $q=p$ or $q>p$, respectively. 
If $ |\al^{\frac{n}{p}}\pm \be^{\frac{n}{p}}|>\frac{2|\al|^{\frac{n}{p}}}{p}$, then 
\begin{align}\label{pq|n}
\frac{|\al^{n}\pm \be^{n}|}{|\al^{\frac{n}{q_0}}\pm \be^{\frac{n}{q_0}}|}&\geq 
\begin{cases}
\frac{Q_n}{p}|\al|^{\frac{n}{p}(1-\frac{1}{q})} & {\rm if} \ q_0=q\\
\frac{Q_n}{p}|\al|^{\frac{n}{p}(1-\frac{1}{p})} & {\rm if} \ q_0=p^2.
\end{cases}
\end{align}
Let $3|n$. Then 
\begin{align}\label{3q|n}
\begin{split}
&\frac{|\al^{n}\pm \be^{n}|}{|\al^{\frac{n}{3}}\pm \be^{\frac{n}{3}}|}\geq 0.56|\al|^{\frac{2n}{3}} \\
 {\rm and} \qquad &\frac{|\al^{n}\pm \be^{n}|}{|\al^{\frac{n}{p^l}}\pm \be^{\frac{n}{p^l}}|}\geq 
\frac{0.56Q_{\frac{n}{3}}}{3}|\al|^{\frac{2n}{3}(1-\frac{1}{p^l})} \quad {\rm if} \ 3p^l|n, p\neq 3
\end{split}
\end{align}
where $p$ is odd in case of $\al^n+\be^n$. 

\noindent
Let $n\geq 10$ be even and $|\al|\geq 3$. For $2p_0|n$ with $p_0=p^l, l\geq 1$, we have 
\begin{align}\label{2|n}
\frac{|\al^{n}-\be^{n}|}{|\al^{\frac{n}{p_0}}-\be^{\frac{n}{p_0}}|}\geq
 \begin{cases}
\frac{7}{8}\min(Q_{\frac{n}{2}}, n+1)|\al|^{\frac{n}{2}(1-\frac{1}{p_0})}  \quad & {\rm if} \quad p_0=p^l, p>2\\
\frac{7}{8}\min(Q_n, n+1)|\al|^{\frac{n}{2}-\frac{n}{2^l}} \quad &{\rm if} 
\quad p_0=2^l.
\end{cases} 
\end{align}
\end{lemma}

\begin{proof}
Let $p\neq q$ be primes dividing $n$. Then  
\begin{align*}
\Phi_{p}(\alpha^{\frac{n}{p}}, \beta^{\frac{n}{p}}) = \frac{\al^n-\be^n}{\al^{\frac{n}{p}}-\be^{\frac{n}{p}}} \quad 
{\rm and} \quad 
\Phi_{pq}(\alpha^{\frac{n}{pq}}, \beta^{\frac{n}{pq}}) = \frac{(\al^n-\be^n)(\al^{\frac{n}{pq}}-\be^{\frac{n}{pq}})}
{(\al^{\frac{n}{p}}-\be^{\frac{n}{p}})(\al^{\frac{n}{q}}-\be^{\frac{n}{q}})} 
\end{align*}
are integers divisible by primitive divisors of $U_n$.  Further for odd $p, q$ with $pq|n$, we see that 
\begin{align*}
\Phi_{2p}(\alpha^{\frac{n}{p}}, \beta^{\frac{n}{p}}) = \frac{\al^n+\be^n}{\al^{\frac{n}{p}}+\be^{\frac{n}{p}}}
 \quad {\rm and} \quad 
\Phi_{2pq}(\alpha^{\frac{n}{pq}}, \beta^{\frac{n}{pq}}) = \frac{(\al^n+\be^n)(\al^{\frac{n}{pq}}+\be^{\frac{n}{pq}})}
{(\al^{\frac{n}{p}}+\be^{\frac{n}{p}})(\al^{\frac{n}{q}}+\be^{\frac{n}{q}})} 
\end{align*}
are integers divisible by primitive divisors of $V_n$. Since primitive divisors of $U_n$ and 
$V_n$ are at least $Q_n$, we obtain 
\begin{align}\label{pqQ}
\frac{|\al^{n}\pm \be^{n}|}{|\al^{\frac{n}{q}}\pm \be^{\frac{n}{q}}|}\geq 
\frac{Q_n|\al^{\frac{n}{p}}\pm \be^{\frac{n}{p}}|}{|\al^{\frac{n}{pq}}\pm \be^{\frac{n}{pq}}|}
\end{align}
where as before, $p, q$ are odd when we consider $\al^n+\be^n$.  We have 
\begin{align*}
\frac{\al^p\pm \be^p}{\al\pm \be}&=\sum^{p-1}_{i=0}\al^{p-1-i}(\mp \be)^i
=p\be^{p-1}+\sum^{p-2}_{i=0}(i+1)(\al^{p-1-i}(\mp \be)^i-\al^{p-2-i}(\mp \be)^{i+1})\\
&=p\be^{p-1}+(\al \pm \be)\sum^{p-2}_{i=0}(i+1)\al^{p-2-i}(\mp \be)^i
\end{align*}
and 
\begin{align*}
\frac{\al^p\pm \be^p}{\al\pm \be}&=\sum^{p-1}_{i=0}\al^{p-1-i}(\mp \be)^i
=\be^{p-1}+\sum^{\frac{p-1}{2}-1}_{i=0}(\al^{p-1-2i}(\mp \be)^{2i}+\al^{p-2-2i}(\mp \be)^{2i+1})\\
&=\be^{p-1}+(\al\mp \be)\sum^{\frac{p-1}{2}-1}_{i=0}\al^{p-2-2i}\be^{2i}. 
\end{align*}

In general, let $p|n$ and further $p$ is odd in case of $\al^n+\be^n$.  Assume 
that $|\al^{n/p}\pm \be^{n/p}|\leq \frac{2|\al|^{n/p}}{p}$. Then 
\begin{align*}
\big{|}\frac{\al^n\pm \be^n}{\al^{n/p}\pm \be^{n/p}}\big{|}&=
|p\be^{\frac{n(p-1)}{p}}+(\al^{n/p} \pm \be^{n/p})\sum^{p-2}_{i=0}(i+1)\al^{(p-2-i)n/p}(\mp \be^{n/p})^i|\\
&\geq p|\al|^{\frac{n(p-1)}{p}}-\frac{2|\al|^{n/p}}{p}|\al|^{n(p-2)/p}\sum^{p-2}_{i=0}(i+1)=|\al|^{n(p-1)/p}.
\end{align*}
Hence $|\al^n\pm \be^n|\geq |\al^{n/p}\pm \be^{n/p}||\al|^{n(p-1)/p}|\geq \al|^{n(p-1)/p}$ giving 
the first inequality of \eqref{p|n<}. Let $|\al^{n/p}\pm \be^{n/p}|>\frac{2|\al|^{n/p}}{p}$. Then 
\eqref{p|n<} follows from the fact that $\Phi_p(\al^{n/p}, \be^{n/p})$ and $\Phi_{2p}(\al^{n/p}, \be^{n/p})$ 
are divisible by primitive divisors of $U_n$ and $V_n$, respectively. 

Let $pq|n$ with $q\geq p$. Then \eqref{pq|n}  follows from \eqref{p|n<} by using 
$|\al^{n/q_0}+\be^{n/q_0}|\leq 2|\al|^{n/q_0}$ and \eqref{pqQ}. 

Let $3|n$. Then the first assertion of \eqref{3q|n} follows from Lemma \ref{x3/x}. The second assertion 
follows since 
\begin{align*}
\left|\frac{\al^n\pm \be^n}{\al^{\frac{n}{p^l}}\pm \be^{\frac{n}{p^l}}}\right|=
\left|\frac{\al^n\pm \be^n}{\al^{\frac{n}{3}}\pm \be^{\frac{n}{3}}}
\frac{\al^{\frac{n}{3}}\pm \be^{\frac{n}{3}}}{\al^{\frac{n}{3p^l}}\pm \be^{\frac{n}{3p^l}}}
\frac{\al^{\frac{n}{3p^l}}\pm \be^{\frac{n}{3p^l}}}{\al^{\frac{n}{p^l}}\pm \be^{\frac{n}{p^l}}}\right|
\geq 0.56|\al|^{\frac{2n}{3}}Q_{\frac{n}{3}}\frac{1}{3|\al|^{\frac{2n}{3p^l}}}.
\end{align*}

Let $2|n$ with $n\geq 10$ and $|\al|\geq 3$. From the triangle inequality 
$$
2|\al^{n/2}|\leq |\al^{n/2}-\be^{n/2}|+|\al^{n/2}+\beta^{n/2}|,
$$  
we have either 
$$|\al^{n/2}\pm \be^{n/2}|\geq  \frac{7}{4}|\al|^{\frac{n}{2}} \quad {\rm and} \quad  
|\al^{n/2}\mp \be^{n/2}|\leq  \frac{1}{4}|\al|^{\frac{n}{2}}$$
or $\min(|\al^{n/2}-\be^{n/2}|, |\al^{n/2}+\be^{n/2}|)\geq \frac{1}{4}|\al|^{\frac{n}{2}}$. In the latter case, 
we have  $$|\al^n-\be^n|\geq \frac{1}{16}|\al|^n\geq  \frac{7(n+1)}{8}|\al|^{n/2}$$ 
for $n\geq 10$ and $|\al|\geq 3$ and the assertion \eqref{2|n} follows by using $|\al^t\pm \be^t|\leq 2|\al|^t$ 
for $t\geq 1$. Consider the former case.  Then for $p_0=p^l, l\geq 1$ and further $l>1$ if $p=2$, we have 
\begin{align*}
 &\frac{|\al^{n}-\be^{n}|}{|\al^{p_0}-\be^{p_0}|}=
 \frac{|\al^{n/2}-\be^{n/2}||\al^{n/2}+\be^{n/2}|}{|\al^{\frac{n}{p_0}}-\be^{\frac{n}{p_0}}|}\\
\geq &\begin{cases}
\frac{7}{4}\frac{|\al|^{\frac{n}{2}}|\al^{\frac{n}{2}}+\be^{\frac{n}{2}}|}{2|\al|^{\frac{n}{2^l}}}
=\frac{7|\al|^{\frac{n}{2}-\frac{n}{2^l}}|V_{\frac{n}{2}}|}{8}\geq \frac{7Q_n|\al|^{\frac{n}{2}-\frac{n}{2^l}}}{8}
&|\al^{\frac{n}{2}}-\be^{\frac{n}{2}}|\geq \frac{7}{4}|\al|^{\frac{n}{2}},\\
& p_0=2^l;\\
\frac{7}{4}\frac{|\al|^{\frac{n}{2}}|\al^{\frac{n}{2}}-\be^{\frac{n}{2}}|}{|\al^{\frac{n}{2^l}}-\be^{\frac{n}{2^l}}|}
=\frac{7|\al|^{\frac{n}{2}}|U_{\frac{n}{2}}|}{4|U_{\frac{n}{2^l}}|}\geq \frac{7}{4}Q_{\frac{n}{2}}|\al|^{\frac{n}{2}}, 
&|\al^{\frac{n}{2}}+\be^{\frac{n}{2}}|\geq \frac{7}{4}|\al|^{\frac{n}{2}},\\
& p_0=2^l;\\
\frac{7}{4}\frac{|\al|^{\frac{n}{2}}|\al^{\frac{n}{2}}+\be^{\frac{n}{2}}|}{2|\al|^{\frac{n}{2p_0}}|\al^{\frac{n}{2p_0}}+
\be^{\frac{n}{2p_0}}|}=\frac{7|\al|^{\frac{n}{2}-\frac{n}{2p_0}}|V_{\frac{n}{2}}|}{8|V_{\frac{n}{2p_0}}|}\geq 
\frac{7Q_n|\al|^{\frac{n}{2}-\frac{n}{2p_0}}}{8}, 
& |\al^{\frac{n}{2}}-\be^{\frac{n}{2}}|\geq \frac{7}{4}|\al|^{\frac{n}{2}}, \\
&p>2;\\
\frac{7}{4}\frac{|\al|^{\frac{n}{2}}|\al^{\frac{n}{2}}-\be^{\frac{n}{2}}|}{2|\al|^{\frac{n}{2p_0}}
|\al^{\frac{n}{2p_0}}-\be^{\frac{n}{2p_0}}|}=\frac{7|\al|^{\frac{n}{2}-\frac{n}{2p_0}}
|U_{\frac{n}{2}}|}{8|U_{\frac{n}{2p_0}}|}\geq 
\frac{7Q_{\frac{n}{2}}|\al|^{\frac{n}{2}-\frac{n}{2p_0}}}{8}, 
&|\al^{\frac{n}{2}}+\be^{\frac{n}{2}}|\geq \frac{7}{4}|\al|^{\frac{n}{2}}, \\
&p>2.
\end{cases}
\end{align*}
This gives the assertion \eqref{2|n} since $|\al|\geq 4$, and $Q_n\geq Q_{n/2}\geq n/2-1$.  
\end{proof}

\begin{lemma}\label{p-fp}
Let $p$ be a prime and put 
$$f_p:=\frac{3.9}{p-1}+\frac{1.46\log 3p}{p}.$$
For $n\in \{p, 2p\}$, let $\log |\al|\geq 1.35(1-f_p)n$ and $\be/\al$ is not a root of unity. Then 
\begin{align}\label{1-fp}
\begin{split}
\log |\al^p\pm \be^p|&>\frac{\log |\al|+\log 2p}{1-f_p} \qquad {\rm for} \quad p>787\\ 
{\rm and} \quad  \log |\al^{2p}+\be^{2p}|&>\frac{2\log |\al|+\log 2p}{1-f_p} \ \quad {\rm for} \quad p>397.
\end{split}
\end{align} 
\end{lemma}

\begin{proof}
We use linear forms in logarithms. Since $r^2+4s<0$, we have $|\al|=|\be|=\sqrt{|s|}$. 
Let $\gamma=\be/\alpha$. Then the minimal polynomial of $\pm \gamma$ is 
$sx^2\mp (r^2+2s)x+s$ and both $\pm \gamma$ and its complex conjugate $1/\pm \gamma$  
have absolute values $1$. Thus,
$$
h(\gamma)=h(-\gamma)=\frac{\log |s|}{2}=\log |\al|.
$$ 
For $n\in \{p, 2p\}$, put  $\delta_n=n(1-\frac{1}{p(1-f_p)})$ and $A_n=1.35(1-f_p)n$. 
Let $\log |\al|\geq A_n$ and $\be/\al$ is not a root of unity. Suppose \eqref{1-fp} does not hold. Then from 
$|\al^n\pm \be^n|=|\al|^n|1\pm \gamma^n|$, we get 
\begin{align*}
\log |1\pm \gamma^n|+\delta_n\log |\al|-\frac{\log 2p}{1-f_p}<0.
\end{align*}
Observe that $-1=e^{i\pi}$. Hence $|1\pm \gamma^n|=|e^{z}-1|$ with $z\in \{n\log \gamma, n\gamma+\pi i\}$. 
By maximum modulus principle, for a complex number $z$, we have either 
$|e^z-1|\geq \frac{1}{2}$ or $|e^z-1|>\frac{2}{\pi}|z-k'\pi i|$ for some integer $k'$. If 
$|e^z-1|\geq \frac{1}{2}$, then clearly 
$$\log |1\pm \gamma^n|+\delta_n\log |\al|-\frac{\log 2p}{1-f_p}>-\log 2+
\left(p-\frac{1}{1-f_p}\right)A_n-\frac{\log 2p}{1-f_p}>0$$
for $p\geq 397$. Thus $|e^z-1|>\frac{2}{\pi}|z-k\pi i|$ giving 
\begin{align}\label{tempo<}
\delta_n\log |\al|-\frac{\log 2p}{1-f_p}+\log \frac{2}{\pi}+\log |\Lambda|<0
\end{align}
where 
$$\Lambda=n\log \gamma-k\pi i, \quad k\in \Z, |k|\leq m.$$
We use \cite[Theorem A.1.3]{BHV} with $\al=\gamma$ and $\lambda=2.5$. We have $D=1, B=n$,  
$$\rho=e^\lambda, \quad t=\frac{1}{6\pi \rho}-\frac{1}{48\pi \rho(1+\frac{2\pi \rho}{3\lambda})}, \quad 
k=\left(\frac{1}{3}+\sqrt{\frac{1}{9}+2\lambda t}\right)$$
and 
\begin{align*}
\mathscr{H}\leq \mathscr{H}_0:=&\log m+\log\left(\frac{1}{\pi \rho}+\frac{1}{\rho \pi+2A_m}\right)-\log \sqrt{k}
+0.886\\
&+\frac{3\lambda}{2}+\frac{1}{3k}\left(\frac{1}{2\pi \rho}+\frac{1}{0.5\rho \pi+A_m}\right)+0.023, 
\end{align*}
by using \cite[(A.12)]{BHV} and $\log |\al|\geq A_m$, to obtain
\begin{align*}
\log |\Lambda|>-c_m\log |\al|-2\mathscr{H}-2\log \mathscr{H}+0.5\lambda+2\log \lambda-3\log 2
\end{align*}
where 
$$c_m=(8\pi k\rho\lambda^{-1}\mathscr{H}^2_0+0.023)\left(1+\frac{0.5\rho \pi}{A_m}\right).$$
This together with \eqref{tempo<} gives
\begin{align*}
0>(\delta_m-c_m)\log |\al|-\frac{\log 2p}{1-f_p}+\log \frac{2}{\pi}-2\mathscr{H}-2\log \mathscr{H}
+0.5\lambda+2\log \lambda-3\log 2.
\end{align*}
Finally we use $\log |\al|\geq A_m$ to see that the right hand side is an increasing function of $p$ 
and the inequality is not valid for $m=p>787$ and $m=2p>397$. Hence the proof. 
\end{proof}

The next lemma is on the upper bound for prime powers dividing a product of factorials and a product of 
middle binomial coefficients and Catalan numbers.

For positive integers $1<m_1\leq m_2\leq \cdots \leq m_k$, let 
\begin{align}\label{cF}
\mathcal{F}:=\mathcal{F}(m_1, m_2, \cdots, m_k):=m_1!m_2!\cdots m_k!. 
\end{align}
and 
\begin{align}\label{cD}
\mathcal{D}:=\mathcal{D}(m_1, m_2, \ldots, m_k):=\prod^k_{i=1}D_{m_i}, \quad D_{m_i}\in \{C_{m_i}, B_{m_i}\}. 
\end{align}
Let $n_0$ be a positive integer. Recall the definition of $M_{n_0}(\ell)$ given in \eqref{Mn}. The following 
results on the upper bounds for 
\begin{equation*}
M_{n_0}(\mathcal{F})=\log\left(\prod_{\substack{p^{\nu_p}\|\cF \\ 
p\equiv \pm 1\pmod {n_0}}} p^{\nu_p}\right)=
\sum _{\substack{p^{\alpha_p}\|\cF \\ p\equiv \pm 1\pmod {n_0}}}\nu_p\log p.
\end{equation*}
and 
\begin{equation*}
M_{n_0}(\mathcal{D}):=\log\left(\prod_{\substack{p^{\nu_p}\|\cD \\ 
p\equiv \pm 1\pmod {n_0}}} p^{\nu_p}\right)=
\sum _{\substack{p^{\alpha_p}\|\cD \\ p\equiv \pm 1\pmod {n_0}}}\nu_p\log p.
\end{equation*}
are given by \cite[Lemma 9]{LLS-F} and \cite[Lemma 8]{LLS-C}, respectively.  
 
\begin{lemma}\label{MnBDF}
$(a)$ For a positive integer $n_0\geq 30$ and $m_k\geq 3$, we have 
\begin{equation}\label{Mnub}
M_{n_0}(\cF)<\frac{4}{\varphi(n_0)}(1+\log\log n_0)\left(\log \cF -1.4\log m_k\right).
\end{equation}
\noindent 
$(b)$ Let $\mathcal{D}:=\{8, 9, 12, 16, 24, 25, 27\}\cup \{5\leq p\leq 30:  p \ { prime}\}$.  
For $n_0\in \cD$ and $m_k\geq 3$, we have 
 \begin{align}\label{Mnub<30}
M_{n_0}(\cF)\le \frac{2.2}{\varphi(n_0)}\left(\log \cF -1.4\log m_k\right) .
\end{align}
Further for $n_0\in \{8, 12\}$, we have 
 \begin{align}\label{Mnub812} 
M_{n_0}(\cF)\le \begin{cases}
0.23\log \cF & \quad {\rm when} \ n_0=8 \ {\rm and} \ 7\leq m_k<17;\\
0.3\log \cF & \quad {\rm when} \ n_0=8 \ {\rm and} \ 17\leq m_k\leq 47;\\
0.28\log \cF & \quad {\rm when} \ n_0=12 \ {\rm and} \ 7\leq m_k\leq 16.
\end{cases}
\end{align}
\end{lemma}

\begin{lemma}\label{MnBDC}
For $n_0\geq 25$, we have 
\begin{equation}\label{Mnub}
M_{n_0}(\cD)\leq  \begin{cases}
\left(\frac{3.9}{\varphi(n_0)}+2.92\frac{\log 3n_0}{n_0}\right)(\log \cD-\log 2)
& \  {\rm if} \ n_0 \ {\rm is \ even};\\
\left(\frac{3.9}{\varphi(n_0)}+1.46\frac{\log 3n_0}{n_0}\right)(\log \cD-\log 2)  & \ 
{\rm if} \ n_0 \ {\rm is \ odd}.
\end{cases}
\end{equation}
Let $n_0\in \{9, 16, 24\}$ or $5\leq n_0\leq 23$ be a prime.  We have  
\begin{align}\label{Mnub<30}
M_{n_0}(\cD)\le \begin{cases}
 \frac{\delta_0}{\varphi(n_0)}\log \cD, & {\rm if} \quad m_k<1500;\\
\frac{\delta_0}{\varphi(n_0)}\left(\log \cD -\log 2\right), & {\rm if} \quad m_k\geq 1500; \\
\end{cases}
\end{align}
where $\delta_0$ is given by 
\begin{center}
\begin{tabular}{|c||c|c|c|c|c|c|c|} \hline
$n_0$ & $5$ & $7$  & $9$ & $16$ & $24$ & $11\leq n_0\leq 23$    \\ \hline \hline
$\delta_0$ & $2.61$ & $3.19$ & $3.57$ & $2.89$ & $2.746$ & $3.3$   \\ \hline 
\end{tabular}
\end{center}
\end{lemma}

\begin{lemma}\label{mk<N}
Write each $\cD=\prod^k_{i=1}D_{m_i}$ as 
\begin{align*}
\log \cD=\sum_jt_{1j}\log C_j+\sum_jt_{2j}\log B_j
\end{align*}
where 
$$
t_{1j}:=\#\{i: D_{m_i}=C_j\} \qquad  {\rm and} \qquad t_{2j}:=\#\{i: D_{m_i}=B_j\}.$$
Given $n_0$, let  
\begin{align*}
\ep_{1j}:=\frac{M_{n_0}(C_j)}{\log C_j} \quad {\rm and} \quad \ep_{2j}:=\frac{M_{24}(B_j)}{\log B_j}.
\end{align*}
Let $m_k\leq N$.  Suppose 
\begin{align*}
\ep_N:=\underset{t_{1j}+t_{2j}>0, j\leq N}\max \{\ep_{1j}, \ep_{2j}\}.
\end{align*}
Then for $\ep_0\geq \ep_N$, we have 
\begin{align}\label{Mne0}
M_{n_0}(\cD)\leq \ep_0\left(\log \cD-\sum_jt_{1j}\lambda_{1j}-\sum_jt_{2j}\lambda_{2j} \right)
\leq \ep_0\log \cD
\end{align}
where 
$$\lambda_{1j}:=\left(1-\frac{\ep_{1j}}{\ep_0}\right)\log C_j
\quad {\rm and} \quad \lambda_{2j}:=\left(1-\frac{\ep_{2j}}{\ep_0}\right)\log B_{j}.$$
\end{lemma}

\begin{proof}
Observe that $\ep_{1j}=\ep_{2j}=0$ for $2j<Q_{n_0}$.  We have 
\begin{align*}
M_{n_0}(\cD)&=\sum_j(\ep_{1j}t_{1j}\log C_{j}+\ep_{2j}t_{2j}\log B_{j})\\
&=\ep_0\sum_j\left(\left(1-\left(1-\frac{\ep_{1j}}{\ep_0}\right)\right)t_{1j}\log C_{j}\right.\\
& +
\left.\left(1-\left(1-\frac{\ep_{2j}}{\ep_0}\right)\right)t_{2j}\log B_{j}\right)\\
&\leq \ep_0\left(\log \cD-\sum_jt_{1j}\lambda_{1j}-\sum_jt_{2j}\lambda_{2j} \right).  
\end{align*}
Hence the proof. 
\end{proof}

The following result follows from \cite[Lemma 9]{LLS-C}. 

\begin{lemma}\label{m14} 
The function $\frac{\log(C_m/2)}{m}$ is increasing  for $m\geq 7$. Hence, 
\begin{align}\label{m1.3}
\log \left(\frac{B_m}{2}\right)>\log \left(\frac{C_m}{2}\right)>\begin{cases}
m \quad &{\rm for} \quad  m\geq 14;\\
1.36m \quad &{\rm for} \quad  m\geq 400;\\
1.38m &{\rm for} \quad  m\geq 2100.
\end{cases}
\end{align}
\end{lemma}

The following lemma is analogous to \cite[Lemma 3]{LLS-C} and follows from the fact that 
$s<0$, $r^2<-4s$ and $|\al|^2=-s$ when $\al, \be$ are complex conjugates. 

\begin{lemma}\label{rsBD}
Let $\al, ~\be$  be complex conjugates so that $r^2+4s<0$. Let $c_1\leq |\al|\leq c_2$ where $c_1, c_2$ are positive reals.  Then $c^2_1\leq -s\leq c^2_2$ and $1\leq r<2\sqrt{-s}$. 
\end{lemma}

\section{Proof of Theorems \ref{thmF} and \ref{thmC} \label{thm-F}}

We recall that for $n\geq 0$
\begin{align*}
U_n=\frac{\alpha^n-\beta^n}{\alpha-\beta} \quad {\rm and} \quad 
V_n= \alpha^n+\beta^n.
\end{align*}
where $\alpha$ and $\beta$ are the roots of the quadratic equation $\lambda^2-r\lambda-s=0$ and 
$r, s$ are coprime nonzero integers with $r^2+4s\neq 0$. We suppose that $\alpha/\beta$ is not a root of 
unity.  We also recall that we assume $r>0$ and $|\al|\geq |\be|$. 
Put $X_n\in \{|U_n|, |V_n|\}$ and $Y\in  \{\cF, \cD\}$ where 
$$
\cF=m_1!m_2!\cdots m_k! \quad {\rm and} \quad 
\cD=D_{m_1}\cdots D_{m_k}, \quad D_{m_i}\in \{C_{m_i}, B_{m_i}\}$$ 
with $1<m_1\leq \cdots \leq m_k$. 

We consider the equations \eqref{eqn1}-\eqref{eqn4} and assume that one of the equations 
has a solution. Hence the equation 
equation 
\begin{align*}
X_n=Y
\end{align*}
has a solution. As in \cite{LLS-F} and \cite{LLS-C}, for a suitable divisor $n_0$ of $n$, we will compare 
the upper bound of $M_{n_0}(Y)$ given by Lemma \ref{MnBDF} or  \ref{MnBDC} with the 
lower bound of $M_{n_0}(Y)=M_{n_0}(X)$ obtained by using Lemma \ref{primn_0}. 
We will choose a suitable divisor $n_0$ of $n$ such that these bounds contradict each 
other and hence for $n$ with such divisors $n_0$, $X_n=Y$  is not possible.   

We take $n\geq 5$ with $n\notin \{6, 8, 12\}$. Recall that a prime $p\mid X_n$ is 
a primitive divisor of $X_n$ if $p\nmid X_\ell$ for $\ell<n$ and $p\nmid r^2+4s$. It is known that 
primitive divisors of $U_n$ and $V_n$ are congruent to $ \pm 1$ modulo $n$ and $ \pm 1$ modulo $2n$, 
respectively. The solutions when $X_n$ has no primitive divisor were already considered 
in \cite{LLS-F} and \cite{LLS-C}. They are in fact gives by those listed in the statement of Theorems 
\ref{thmF} and \ref{thmC}. 

Thus we assume that $X_n$  has a primitive prime divisor $p$ and so $p\equiv \pm 1\pmod n$.  
Let $P_n$ be the largest primitive divisor of $X_n$.  From equations \eqref{eqn1}-\eqref{eqn2}, 
 we have that $P_n\mid m_k!$  implying $m_k\geq P_n$ and from the equations 
 \eqref{eqn3}-\eqref{eqn4}, we have $P_n|D_{m_k}$ implying  $m_k\geq \frac{P_n+1}{2}$.  
 As defined before, $P_n\geq Q_n$,  the least prime congruent to one of $\pm 1$ modulo $n$. 
 Therefore 
\begin{align}\label{logal}
2|\alpha|^{n}\geq X_n\ge 
\begin{cases}
m_k!\geq P_n!\geq Q_n!, &Y=\cF;\\
C_{m_k}\geq C_{\frac{P_n+1}{2}}\geq C_{\frac{Q_n+1}{2}}, &Y=\cD.
\end{cases}
\end{align}
We obtain by using $t!\ge \sqrt{2\pi t} (t/e)^t>2(t/e)^{t+1/2}$ for $t>1$ and $\log C_t/2>t$ for 
$t\geq 14$ by Lemma \ref{m14} that 
\begin{align}\label{Qn>30}
\begin{split}
\log X_n\geq 
\begin{cases}
\log Q_n!\geq \left(Q_n+\frac{1}{2}\right)(\log Q_n-1), &Y=\cF;\\
\log C_{\frac{Q_n+1}{2}}\geq \frac{Q_n+1}{2}, & Y=\cD, Q_n\geq 14.
\end{cases}
\end{split}
\end{align}
It is easy to observe from above that $\min(Q_n!, C_{\frac{Q_n+1}{2}})> \frac{Q_n+1}{2}$ for 
$Q_n\geq 27$.  

For $n_0\in \{8, 9, 16, 24\}$, or $n_0=q$ an odd prime, we define 
\begin{align}\label{gn}
g(n_0):=\begin{cases}
\frac{2.2}{\varphi(n_0)}, &n_0\in \{8, 9, 16, 24\} \ {\rm or} \ n_0=q\leq 29;\\
\frac{4}{q-1}(1+\log\log q), &n_0=q>29
\end{cases}
\end{align} 
and 
\begin{align}\label{gn}
f(n_0):=\begin{cases}
\frac{\delta_0}{\varphi(n_0)}, &n_0\in \{9, 16, 24\} \ {\rm or} \ n_0=q\leq 29;\\
\frac{3.9}{q-1}+\frac{1.46\log 3q}{q}, &n_0=q>29
\end{cases}
\end{align} 
where $\delta_0$ is given by Lemma \ref{MnBDC}.  We put $h(n_0)=g(n_0), f(n_0)$ 
while considering equations \eqref{eqn1}-\eqref{eqn2} or \eqref{eqn3}-\eqref{eqn4}, respectively. 

Assume that 
\begin{align}\label{Dub}
M_{n_0}(Y)\leq \ep(\log Y-\lambda)= \ep(\log X_n-\lambda). 
\end{align}
for some $0<\ep<1$ and $\lambda\geq 0$. By Lemmas  \ref{MnBDF} and \ref{MnBDC}, 
the inequality \eqref{Dub} is valid with $\ep=h(n_0)$ and $\lambda=0$. 

Let $q^{h+t}\mid n$, where $q$ is a prime and let $h>0, t\geq 0$ are integers such that 
$q^h>4$.  Here we take $q>2$ when $X_n=|V_n|$. 
Taking $n_0=q^h$ and using \eqref{n_0U} and \eqref{n_0V} in Lemma \ref{primn_0}, 
we get a lower bound for $M_{n_0}(X_n)=M_{n_0}(Y)$ which we compare with \eqref{Dub}. 
We obtain 
\begin{align*}
\ep(\log X_n-\lambda) \geq \log \frac{|\al^n\pm \be^n|}{|\al^{\frac{n}{q^{1+t}}}\pm \be^{\frac{n}{q^{1+t}}}|}-\log(q^{t+1})
\end{align*}
where the signs $-$ and $+$ are according as $X_n=|U_n|, |V_n|$, respectively.  Put 
$$Z:=Z_{n, q, h, t}:=\frac{|\al^n\pm \be^n|}{|\al^{\frac{n}{q^{1+t}}}\pm \be^{\frac{n}{q^{1+t}}}|} \quad 
{\rm and} \quad Z_1=\frac{X_n}{Z}.$$ Then 
\begin{align}\label{Dcmp}
\begin{split}
0&\geq (1-\ep)\log X_n-\log Z_1-\log(q^{t+1})+\lambda\ep\\
&\geq (1-\ep)\log Z-\ep\log Z_1-\log(q^{t+1})+\lambda\ep.
\end{split}
\end{align}
From $X_n=Y$, this also gives 
\begin{align}\label{al-Y}
\frac{n\log |\al|}{q^{1+t}}\geq \log Z_1-\log 2\geq (1-\ep)\log Y-\log(2q^{t+1})+\lambda \ep. 
\end{align}
Let $c_1, c_2>0$ be such that 
\begin{align}\label{c1c2}
\log Z\geq \log c_1+\frac{c_2n\log |\al|}{q^{t+1}}\geq \log c_1+c_2(\log Z_1-\log 2).
\end{align}
From \eqref{Dcmp} and \eqref{al-Y}, writing $\log X_n=\log Z+\log Z_1$, we get 
\begin{align*}
0\geq &(1-\ep)(\log c_1+c_2(\log Z_1-\log 2))-\ep\log Z_1-\log(q^{t+1})+\lambda\ep\\
\geq & (1-\ep)(\log c_1-c_2\log 2)-\log(q^{t+1})+\lambda\ep+((1-\ep)c_2-\ep)\log Z_1\\
\geq & (1-\ep)(\log c_1-c_2\log 2)-\log(q^{t+1})+\lambda\ep\\
&+((1-\ep)c_2-\ep)\left\{(1-\ep)\log Y-\log(q^{t+1})+\lambda \ep\right\}\\
\geq & (1-\ep)\left\{\log c_1+\log 2-(c_2+1)(\log(2q^{t+1})-\lambda\ep)+((1-\ep)c_2-\ep)\log Y\right\}
\end{align*}
implying 
\begin{align}\label{Yc12}
\log Y\leq \frac{(c_2+1)(\log(2q^{t+1})-\lambda\ep)-\log (2c_1)}{c_2-(c_2+1)\ep}=
\frac{\log(2q^{t+1})-\lambda\ep-\frac{\log (2c_1)}{c_2+1}}{1-\frac{1}{c_2+1}-\ep}
\end{align}
since $\ep<1$. 
By Lemmas  \ref{MnBDF} and \ref{MnBDC},  we can take $\ep=h(n_0)$ and $\lambda=0$ in \eqref{Dub}.  
Hence \eqref{Dcmp}, \eqref{al-Y} and \eqref{Yc12} are valid with 
$\ep=h(n_0)$ and $\lambda=0$. We now consider different possibilities for $n$. 

\subsection{$n$ is prime} 

Let $q\geq 5$ be a prime. We consider $n=q$ and $X_n=|U_n|$.  
Taking $n_0=n=q, t=0$ and $\ep=h(q), \lambda=0$ in \eqref{Dcmp} ($Z_1=1$ here), 
and using $Q_q\geq 2q-1$ in \eqref{Qn>30}, we obtain 
\begin{align*}
1\geq \frac{(1-h(q))\log |U_q|}{\log q}\geq \begin{cases}
(1-g(q))(2q-1) & {\rm if} \ Y=\cF\\
\frac{(1-f(q))\log C_{\frac{Q_q+1}{2}}}{\log q} & {\rm if} \ Y=\cD, q<31\\
\frac{(1-f(q))q}{\log q} & {\rm if} \ Y=\cD, q\geq 31. 
\end{cases}
\end{align*}
We find that the last quantity is always $>1$ which is a contradiction. 
Thus $n\geq 5$ is composite when $X_n=|U_n|$. 

Let $X_n=|V_n|$ and $n=q>787$. Taking $n_0=n=q, t=0$ 
and $\ep=h(q), \lambda=0$ in \eqref{al-Y}, we get 
$$\frac{\log|\al|+\log 2q}{1-h(q)}\geq \log |V_q|.$$
If $\log |\al|\geq 1.35(1-f(q))q\geq 1.35(1-h(q))q$, then we get a contradiction by 
Lemma \ref{p-fp} for $q$. Therefore $\log |\al|<1.35(1-f(q))q$. Hence 
$$1.35q+\frac{\log 2q}{1-f(q)}> 
\log |V_q|\geq \begin{cases}
(2q-\frac{1}{2})(\log (2q-1)-1) & {\rm if} \ Y=\cF \\
\log C_{q}\geq 1.37q & {\rm if} \ Y=\cD
\end{cases}
$$
by \eqref{Qn>30}, $Q_q\geq 2q-1$ and $\log C_q\geq qC_{787}/787>1.37q$ 
by using Lemma \ref{m14}.  This is a contradiction.  Thus $n=q\leq 787$ when $X_n=|V_n|$. 

\subsection{$q|n$ with $q\geq 5$ and $|\al^{\frac{n}{q}}\pm \be^{\frac{n}{q}}|\leq \frac{2|\al|^{\frac{n}{q}}}{q}$}

Let $q\geq 5$ be a prime and $q|n$. Assume that 
$|\al^{\frac{n}{q}}\pm \be^{\frac{n}{q}}|\leq \frac{2|\al|^{\frac{n}{q}}}{q}$. 
We take $n_0=q, t=0$ and $\ep=h(q), \lambda=0$ in \eqref{Yc12}. 
By \eqref{p|n<}, we have $c_1=1$ and $c_2=q-1$ and hence \eqref{Yc12} gives 
\begin{align*}
0\geq & (1-\frac{1}{q}-h(q))\log Y-(\log 2q-\frac{\log 2}{q})\\
\geq & \begin{cases}
(1-\frac{1}{q}-g(q))\log Q_{q}!-\log 2q+\frac{\log 2}{q} & {\rm if} \ Y=\cF, q<31\\
(1-\frac{1}{q}-f(q))\log C_{\frac{Q_{q}+1}{2}}-\log 2q+\frac{\log 2}{q} & {\rm if} \ Y=\cD, q<31\\
(1-\frac{1}{q}-h(q))q-\log 2q+\frac{\log 2}{q} & {\rm if} \ q\geq 31. 
\end{cases}
\end{align*}
We find that the last quantity is $>0$ except when $q\in \{5, 7\}$ and $Y=\cD$. Let 
$q\in \{5, 7\}$ and $Y=\cD$.  Then 
\begin{align}\label{<a^n/q}
\log C_{m_k}\leq \log \cD \leq \frac{\log 2q-\frac{\log 2}{q}}{1-\frac{1}{q}-f(q)}
\end{align}
implying $m_k\leq 13, 8$ according as $q=5, 7$, respectively. 
By  Lemma \ref{mk<N} with $(n_0, N)=(5, 13), (7, 8)$, we have 
\eqref{Dub} with $\lambda=0$ and $\ep=0.55, 0.43,$ respectively, in turn gives 
\eqref{<a^n/q} with $f(q)=0.55, 0.43$ according as $q=5, 7$, respectively. Then we 
get $m_k\leq 9, 6$ according as $q=5, 7$, respectively. Since 
$2m_k\geq Q_q$, we further obtain $q=5$ and $6\leq m_k\leq 9$. Since $5|n$ and 
$n\neq 5$, we observe that $X_n=\cD$ is divisible, by primitive divisors of both $U_5$ 
and $U_n$ when $X_n=|U_n|$, and primitive divisors of both $V_5$ and $V_n$ when 
$X_n=|V_n|$ and these are primes $\equiv \pm 1\pmod{5}$.  However $11$ is the only 
prime $\equiv \pm 1\pmod{5}$ dividing $C_j$ or $B_j$ with $6\leq j\leq 9$. This is a 
contradiction. Thus $|\al^{\frac{n}{q}}\pm \be^{\frac{n}{q}}|>\frac{2|\al|^{\frac{n}{q}}}{q}$ for 
each $q|n, q\ge 5$.  

\subsection{$q^2|n$ with $q\geq 3$}

Let $9|n$. We take $n_0=9=3^2, t=0$ and $\ep=h(9), \lambda=0$ in \eqref{Yc12}. By 
\eqref{3q|n}, we have $c_1=0.56$ and $c_2=2$ and hence \eqref{Yc12} gives 
\begin{align}\label{9fd}
\frac{3\log 6-\log 1.12}{2-3h(9)}\geq \log Y\geq 
\begin{cases}
m_k! & {\rm if} \ Y=\cF\\
C_{m_k} & {\rm if} \ Y=\cD.
\end{cases}
\end{align}
This is a contradiction for $m_k>8, 21$ according as $Y=\cF, \cD$, respectively. Since 
$2m_k\geq Q_9=17$, we have   $Y=\cD$ and $9\leq m_k\leq 21$. 
By  Lemma \ref{mk<N} with $(n_0, N)=(9, 21)$, we have 
\eqref{Dub} with $\lambda=0$ and $\ep=0.594$ which in turn gives 
\eqref{9fd} with $f(q)=0.594$. This is a contradiction since $m_k\geq 9$. 

Let $q^2|n$ with $q\geq 5$. Then $Q_n\geq 2q^2-1$. We take $n_0=q, t=1$ and 
$\ep=h(q), \lambda=0$ in \eqref{Yc12}. Since 
$|\al^{\frac{n}{q}}\pm \be^{\frac{n}{q}}|>\frac{2|\al|^{\frac{n}{q}}}{q}$,  
from \eqref{p|n<}, we have 
$c_1=\frac{Q_n}{q}\geq \frac{Q_{q^2}}{q}$ and $c_2=q-1$ and hence \eqref{Yc12} 
gives 
\begin{align*}
0\geq &(1-\frac{1}{q}-h(q))\log Y-\left(\log (2q^2)-\frac{\log \frac{2Q_{q^2}}{q}}{q}\right)\\
\geq & \begin{cases}
(1-\frac{1}{q}-g(q))\log Q_{q^2}!-\log (2q^2)+ \frac{\log (2Q_{q^2})-\log q}{q} & {\rm if} \ Y=\cF, q<31\\
(1-\frac{1}{q}-f(q))\log C_{\frac{Q_{q^2}+1}{2}}-\log (2q^2)+\frac{\log (2Q_{q^2})-\log q}{q} & 
{\rm if} \ Y=\cD, q<31\\
(1-\frac{1}{q}-h(q))q^2-\log (2q^2)+\frac{\log (4q-\frac{2}{q})}{q} & {\rm if} \ q\geq 31. 
\end{cases}
\end{align*}
using $f(q)\geq g(q)$, \eqref{Qn>30} and $Q_{q^2}\geq 2q^2-1$.  We find that the 
above inequality is not valid. This is a contradiction.  

\subsection{$pq|n$ with $q\geq 5$ and $3\leq p<q$}

Let $pq|n$ with $3\leq p<q$. Then $Q_n\geq Q_{pq}\geq 2pq-1\geq 6q-1$.  
We take $n_0=q, t=0$ and $\ep=h(q), \lambda=0$ in \eqref{Yc12}.  
Since $|\al^{\frac{n}{p}}\pm \be^{\frac{n}{p}}|>\frac{2|\al|^{\frac{n}{p}}}{p}$,  
from \eqref{p|n<}, we have $c_1=\frac{Q_{pq}}{p}\geq 2q-\frac{1}{3}$ and 
$c_2=\frac{q-1}{p}\geq 1+\frac{1}{q-2}$ since $q\geq p+2$ and hence \eqref{Yc12} 
gives 
\begin{align*}
0\geq & \left(1-h(q)-\frac{1}{c_2+1}\right) \log Y-\left(\log 2q-\frac{\log 2(2q-\frac{1}{3})}{c_2+1}\right)\\
\geq & \left(1-h(q)-\frac{1}{2+\frac{1}{q-2}}\right)\log Y-\left(\log 2q-\frac{\log (4q-
\frac{2}{3}}{2+\frac{1}{q-2}}\right)\\
\geq & \begin{cases}
(1-h(q)-\frac{29}{59})3q-\left(\log 2q-\frac{29\log (4q-\frac{2}{3}}{59}\right) & 
{\rm if} \ q\geq 31\\
\left(1-g(q)-\frac{1}{2+\frac{1}{q-2}}\right)\log (6q-1)!-\left(\log 2q-\frac{\log (4q-
\frac{2}{3}}{2+\frac{1}{q-2}}\right) & {\rm if} \ Y=\cF, q<31\\
\left(1-f(q)-\frac{1}{2+\frac{1}{q-2}}\right)\log C_{3q}-\left(\log 2q-\frac{\log (4q-
\frac{2}{3}}{2+\frac{1}{q-2}}\right) & {\rm if} \ Y=\cD, 7<q<31\\
\left(1-f(q)-\frac{1}{2+\frac{1}{q-2}}\right)\log C_{m_k}-\left(\log 2q-\frac{\log (4q-
\frac{2}{3}}{2+\frac{1}{q-2}}\right) & {\rm if} \ Y=\cD, q=5, 7
\end{cases}
\end{align*}
by using \eqref{logal}. We find that this is a contradiction except when 
$Y=\cD$ and $q\in \{5, 7\}$.  Let $Y=\cD$ and $q\in \{5, 7\}$.  Since 
$3\leq p<q$, we  have $(p, q)\in \{(3, 5), (3, 7), (5, 7)\}$.  Further for 
$q=7$, from $Y=\cD\geq C_{m_k}$, we obtain $m_k\leq 64$. 

Let $(p, q)=(5, 7)$. Then $2m_k\geq Q_n\geq Q_{35}=71$ implies $35\leq m_k\leq 64$. 
We use Lemma \ref{mk<N} with $n_0=5, N=64$. We have  \eqref{Dub} with 
$\ep=0.532$ and 
$$\lambda \ep \geq \ep\min_{35\leq j\leq 64}(\min(\lambda_{1j}, \lambda_{2j}))>9.$$ 
As before, taking $n_0=7, t=0$,  we obtain  \eqref{Yc12} with  
$c_1=\frac{71}{5}, c_2=\frac{6}{5}$ and $\ep=0.532, \lambda \ep>9$. Hence 
\begin{align*}
\log C_{m_k}\leq \log \cD<\frac{\log 14-9-\frac{\log \frac{142}{5}}{\frac{11}{5}}}{
\frac{6}{11}-0.532}<0
\end{align*}
which is a contradiction. 

Let $(p, q)\in \{(3, 5), (3, 7)\}$. We have from \eqref{3q|n} and $Q_{n/3}\geq Q_q$ that 
\begin{align*}
\frac{|\al^n\pm \be^n|}{|\al^{\frac{n}{q}}\pm \be^{\frac{n}{q}}|}=&
\frac{|\al^n\pm \be^n|}{|\al^{\frac{n}{3}}\pm \be^{\frac{n}{3}}|}
\frac{|\al^{\frac{n}{3}}\pm \be^{\frac{n}{3}}|}{|\al^{\frac{n}{3q}}\pm \be^{\frac{n}{3q}}|}
\frac{|\al^{\frac{n}{3q}}\pm \be^{\frac{n}{3q}}|}{|\al^{\frac{n}{q}}\pm \be^{\frac{n}{q}}|}\\
\geq &0.56|\al|^{\frac{2n}{3}}Q_{\frac{n}{3}}\frac{1}{3|\al|^{\frac{2n}{3q}}}\geq 
\frac{0.56Q_q}{3}|\al|^{\frac{2(q-1)n}{3q}}.  
\end{align*}
We take $n_0=q, t=0$ and $\ep=f(q), \lambda=0$ in \eqref{Yc12}. 
By \eqref{p|n<}, we have $c_1=\log \frac{0.56Q_q}{3}$ and 
$c_2=\frac{2(q-1)}{3}$ since $q\geq p+2$ and hence \eqref{Yc12}  gives 
\begin{align}\label{15C}
\frac{\log 2q-\frac{c_1+\log 2}{c_2+1}}{1-f(q)-\frac{1}{c_2+1}}\geq \log \cD\geq \log C_{m_k}.
\end{align}
This is a contradiction for $m_k\geq 23, 18$ according as $q=5, 7$, respectively. Since 
$2m_k\geq Q_{3q}$, we get $q=5$ and further $15\leq m_k\leq 22$. We also find that 
$\log C_{10}C_{15}$ exceeds the left hand side of \eqref{15C} and hence 
$\cD<C_{10}C_{15}$. Thus either $\cD=D_{m_k}$ for some $15\leq m_k\leq 22$ or 
$\cD=D_1\cdots D_{m_{k-1}}D_{m_k}$ with $m_{k-1}\leq 9<15\leq m_k\leq 22$. 
Suppose $m_k\neq 17$. We use Lemma \ref{mk<N} with $n_0=5, N=22$. 
We have $t_{1j}+t_{2j}=0$ for $j=17$ in Lemma \ref{mk<N} and hence we have  
\eqref{Dub} with $\ep=0.608$ $\lambda=0$. Putting $0.61$ in place of $h(5)$ in \eqref{Dub}, we 
get the inequality \eqref{15C} with $0.608$ in place of $f(q)=f(5)$. This gives a contradiction 
for $m_k\geq 15$. Thus $m_k=17$ and hence $M_{15}(\cD)\leq M_{15}(B_{17})=\log 29+\log 31$. 
On the other hand, by \eqref{n_0U} in Lemma \ref{primn_0} with $n_0=15, t=0$ along 
with \eqref{3q|n} gives 
\begin{align*}
\exp(M_{15}(X_n))\geq \frac{1}{5}\frac{|\al^n\pm \be^n|}{|\al^{\frac{n}{3}}\pm \be^{\frac{n}{3}}|}
\frac{|\al^{\frac{n}{15}}\pm \be^{\frac{n}{15}}|}{|\al^{\frac{n}{5}}\pm \be^{\frac{n}{5}}|}
\geq \frac{0.56|\al|^{\frac{2n}{3}}}{5}\frac{1}{3|\al|^{\frac{2n}{15}}}>0.037|\al|^{\frac{8n}{15}}. 
\end{align*}
Thus 
\begin{align*}
\log 29+\log 31\geq M_{15}(\cD)=M_{15}(X_n)>\log 0.037+\frac{8n\log |\al|}{15}. 
\end{align*}
We have \eqref{al-Y} with $n_0=q=5, t=1, \ep=f(5)=2.61/4, \lambda=0$ which together with 
 $Y=\cD\geq C_{17}$ gives 
\begin{align*}
0>&\frac{8}{3}((1-f(5))\log C_{17}-\log 10)+\log 0.037-\log 29-\log 31. 
\end{align*}
This is a contradiction. 
 
\subsection{$n$ even and $X_n=|V_n|$}

Let $n$ be even and $X_n=|V_n|$. From \cite{LLS-F} and \cite{LLS-C}, we may assume 
$Y=\cD$ and $2||q$. Further from previous subsections, we need to consider $n=2q$ with 
$q\geq 5$. Let $n=2q$ with $q>397$. Taking $n_0=q$ and $t=0$ in \eqref{al-Y}, we get 
$$\frac{2\log|\al|+\log 2q}{1-f(q)}\geq \log |V_{2q}|=|\al^{2q}+\be^{2q}|.$$
If $\log |\al|\geq 1.35(1-f(q))2q$, then we get a contradiction by 
Lemma \ref{p-fp} for $2q$. Therefore $\log |\al|<2.7(1-f(q))q$. 
Observe that primitive divisors of $V_{2q}$ are primitive divisors of $U_{4q}$ and hence 
$P_{2q}|V_{2q}$ with $P_{2q}\geq Q_{4q}\geq 4q-1$ giving $2m_k>4q-1$. Thus  
$|V_{2q}|\geq C_{m_k}\geq C_{2q}\geq 2qC_{2\cdot 397}/(2\cdot 397)>2.74q$  
by Lemma \ref{m14}.  Hence 
$$2.74q<|V_{2q}|\leq \frac{2\log|\al|+\log 2q}{1-f(q)}<2.7q+\frac{\log 2q}{1-f(q)}. 
$$
This is a contradiction.  Thus $n=2q$ with $q\leq 397$ when $X_n=|V_n|=\cD$ 
and the proof of Theorem \ref{thmF} for the equation \eqref{eqn2} and Theorem \ref{thmC} 
for the equation \eqref{eqn4} are complete. 

\subsection{$n$ even and $X_n=|U_n|$}

Let $n$ be even and $X_n=|U_n|$. From previous subsections, we need to consider 
$n$ with $n=2^\nu q$, $q\geq 5, \nu\geq 1$ or $16|n$ or $24|n$. Let $Y=\cF$. Then 
from $m_k\geq Q_n\geq n-1$. From \eqref{logal}, we have 
$$\log |\al|\geq \frac{(n-\frac{1}{2})(\log (n-1)-1)}{n}>\log 3 \quad {\rm for} \ n\geq 10.$$
Thus $|\al|>3$ when $Y=\cF$. 

Let $32|n$. Let $Y=\cD$. Since $Q_n\geq n-1$, we have  
$m_k\geq \frac{Q_{n}+1}{2}\geq \frac{n}{2}$ and so that 
$$\frac{\log \cD}{n}\geq \frac{\log C_{\frac{n}{2}}}{n}\geq \frac{\log C_{16}}{32}$$
by Lemma \ref{m14}.  Suppose $|\al|<3$. Taking $n_0=16=2^4, t=1$, we have  
$\ep=f(16), \lambda=0$ in \eqref{Dub} and hence in \eqref{al-Y} which gives 
\begin{align*}
0\geq &  \frac{(1-f(16))\log \cD-\log 8-\frac{n\log 3}{4}}{n}\\
\geq &(1-f(16))\frac{\log C_{16}}{32}-\frac{\log 3}{4}-\frac{\log 8}{32}>0.
\end{align*}
This is a contradiction and hence $|\al|\geq 3$. 

Now we take $Y\in \{\cF, \cD\}$ again. Taking $n_0=16=2^4, t=1$, we have  
$\ep=h(16), \lambda=0$ in \eqref{Dub} and hence in \eqref{Yc12}. From \eqref{2|n} with 
$p_0=2^2$, we have $2c_1=\frac{7}{4}\times 31=54.25$ and 
$c_2=1$ and hence \eqref{Yc12} gives 
\begin{align*}
\frac{\log 8-\frac{\log 54.25}{2}}{\frac{1}{2}-h(16)}\geq \log Y
\geq & \begin{cases}
\log 31! & {\rm if} \ Y=\cF\\
\log C_{15} & {\rm if} \ Y=\cD
\end{cases}
\end{align*}
by using \eqref{logal}. This is a contradiction. 

Thus $n\in \{16, 24, 48\}$ or $n=2^\nu q$ with $q\geq 5$ and $1\leq \nu\leq 4$. 

\subsection{$n$ even with ord$_2(n)\leq 4$}

Now we modify the inequality \eqref{al-Y}. Using $|(\al^l-\be^l)/(\al-\be)|\leq l|\al|^l$, we 
have from \eqref{Dcmp}
\begin{align}\label{al-ev}
(\frac{n}{q^{1+t}}-1)\log |\al| \geq \log Z_1-\log \frac{n}{q^{1+t}} 
\geq (1-\ep)\log Y-\log n+\lambda \ep. 
\end{align}
This in turn gives 
\begin{align}\label{Yev}
\log Y\leq 
\frac{\log n-\lambda\ep-\frac{\log (\frac{nc_1}{q^{t+1}})}{\frac{c_2}{1-\frac{q^{t+1}}{n}}+1}}
{1-\frac{1}{\frac{c_2}{1-\frac{q^{t+1}}{n}}+1}-\ep}.
\end{align}

\subsection{$n=2^\nu q$ with $q\geq 5$ and $1\leq \nu\leq 4$}

Let $n=2^\nu q$ with $q\geq 5$ and $1\leq \nu\leq 4$.  Then $n\geq 10$ and  
$Q_n\geq n-1$. Suppose $|\al|<3$. Then $Y=\cD$ and 
$$\frac{\log \cD}{n}\geq \frac{\log C_{\frac{n}{2}}}{n}\geq \frac{\log C_{q}}{2q} \quad {\rm for} \ n\geq 14$$
by Lemma \ref{m14}. Taking $n_0=q, t=0$,  we have  
$\ep=f(q), \lambda=0$ in \eqref{Dub} and in \eqref{al-ev} which gives  
\begin{align*}
0\geq & \frac{(1-f(q))\log \cD-\log n-(\frac{n}{q}-1)\log 3}{n}\\
\geq &\begin{cases}
\frac{(1-f(q))\log C_{31}}{62}-\frac{\log 2q}{2q}-\frac{\log 3}{q}
& {\rm if} \ q\geq 31, n\geq 2q\\
\frac{(1-f(q))\log C_{q}}{2q}-\frac{\log 2q}{2q}-(\frac{1}{q}-\frac{1}{16q})\log 3& 
{\rm if} \ 11\leq q<31, n\geq 2q\\
\frac{(1-h(q))\log C_{m_k}}{n}-\frac{\log n}{n}-(\frac{1}{q}-\frac{1}{n})\log 3 & 
{\rm if} \ 5\leq q\leq 7, n=2^\nu q. 
\end{cases}
\end{align*}
We find that that it is a contradiction for $q\geq 11$. For $q\in \{5, 7\}$ and $1\leq \nu\leq 4$, 
we obtain the upper bound of $m_k$ given by the following table. 
\begin{center}
\begin{tabular}{|c||c|c|c|c|c} \hline
$\nu $ & $1$ & $2$  & $3$ & $4$     \\ \hline 
$q=5$ & $10$ & $16$ & $27$ & $47$   \\ \hline 
$q=7$ & $8$ & $13$ & $21$ & $36$   \\ \hline 
\end{tabular}
\end{center}
Since $2m_k\geq Q_n$, we need to consider $n\in \{10, 20, 40, 80, 14\}$. Observe that 
 $U_{14}$ is divisible by primitive divisors of $U_{7}$ and $U_{14}$ and these are 
 primes $\equiv \pm 1\pmod{q}$. Since $13$ is the only prime $\equiv \pm 1\pmod{7}$ 
dividing $C_jB_j$ with $j\leq 14$, we have a contradiction when $n=14$. For  
$U_{n}$ with $n\in \{10, 20, 40, 80\}$ and $|\al|<3$, we check that the equation 
\eqref{eqn3} does not hold. 

Thus we have $|\al|\geq 3$. Taking $n_0=q, t=0$,  we have  
$\ep=h(q), \lambda=0$ in \eqref{Dub} and in \eqref{Yev}.  
From \eqref{2|n} with $p_0=q$, we have 
$c_1=\frac{7}{8}\min(Q_{n/2}, n+1)\geq \frac{7}{8}(2q-1)$ and $c_2=\frac{q-1}{2}$ so that
$$\frac{nc_1}{q} \geq  \left(\frac{7}{4}-\frac{7}{8q}\right)n
  \quad {\rm and} \quad 
\frac{c_2}{1-\frac{q}{n}}+1=\frac{q+1-2^{1-\nu}}{2-2^{1-\nu}}\geq q.$$
Hence, by using \eqref{logal}, the inequality \eqref{Yev} gives 
\begin{align*}
0\geq &\left(1-\frac{1}{\frac{c_2}{1-\frac{q}{n}}+1}-h(q)\right)\log Y-
\left(\log n-\frac{\log (\frac{nc_1}{q})}{\frac{c_2}{1-\frac{q}{n}}+1}\right)\\
\geq & \begin{cases}
(1-\frac{1}{q}-h(q))2^{\nu-1}q-\left(1-\frac{1}{q}\right)\log (2^{\nu}q)+\frac{\log\left(\frac{7}{4}-\frac{7}{8q}\right)}{q}, 
&  q\geq 31;\\
(1-\frac{1}{q}-g(q))\log (2^{\nu}q-1)!-\left(1-\frac{1}{q}\right)\log (2^{\nu}q)+
\frac{\log\left(\frac{7}{4}-\frac{7}{8q}\right)}{q}, & \ Y=\cF,  q<31;\\
(1-\frac{1}{q}-f(q))\log C_{2^{\nu-1}q}-\left(1-\frac{1}{q}\right)\log (2^{\nu}q)+\frac{\frac{7}{4}-\frac{7}{8q}}{q}, 
& Y=\cD,\\
&  7<q<31;\\ 
(1-\frac{2-2^{1-\nu}}{q+1-2^{1-\nu}}-f(q))\log C_{m_k}-\log n+ &Y=\cD, \\
\frac{2-2^{1-\nu}}{q+1-2^{1-\nu}}\log (7\cdot 2^{\nu-3}\min(Q_{n/2}, n+1)), 
& q=5, 7;
\end{cases}
\end{align*}
where $n=2^\nu q$ with $1\leq \nu\leq 4$.  For each $1\leq \nu\leq 4$, we get a 
contradiction for $Y=\cF$ and for $q>7$ when $Y=\cD$. Let $Y=\cD$ and $q\in \{5, 7\}$.  
The inequality implies $m_k<\frac{Q_n+1}{2}$ when $n\in \{28, 56, 112\}$ which is 
a contradiction and further $m_k\leq 7, 11, 25, 48, 81$  according as 
$n=14, 10, 20, 40, 80,$ respectively. 

Let $n=14$ and $m_k=7$. Then $13$ is the only prime $\equiv \pm 1\pmod{7}$ 
dividing $C_jB_j$ with $j\leq 14$. Since $U_{14}$ is divisible by primitive divisors of 
$U_7$ and $U_{14}$ which are primes $\equiv \pm 1\pmod{7}$,  we get a contradiction. 

Thus we need to consider $q=5$ and $n\in \{10, 20, 40, 80\}$.  Then $10\leq m_k\leq 81$. 
We now use Lemma \ref{mk<N} with $n_0=5$ and 
$N=81$. We get  \eqref{Dub} with $\ep=\ep_0=0.6501$ and 
$$\lambda \ep=\sum_{j\leq 81}(t_{1j}\ep\lambda_{1j} +t_{2j}\ep\lambda_{2j})=:\mathcal{L}.$$
As before, taking $n_0=5, t=0$,  we have  $\ep=0.6501, \lambda=\mathcal{L}$ in \eqref{Dub} 
and hence from \eqref{Yev}, we obtain for $n\in \{10, 20, 40, 80\}$, 
\begin{align}\label{1020}
\begin{split}
\log \cD\leq &\frac{\log n-
\frac{2-2^{1-\nu}}{6-2^{1-\nu}}\log (7\cdot 2^{\nu-3}\min(Q_{n/2}, n+1))-\mathcal{L}}{
1-\frac{2-2^{1-\nu}}{6-2^{1-\nu}}-0.6501}\\
<&\frac{2.355-\mathcal{L}}{1-\frac{2-2^{\nu-1}}{q+1-2^{1-\nu}}-0.6501}.
\end{split}
\end{align}
We find that $\ep\lambda_{1j}, \ep\lambda_{2j}>2.355$ for $j\geq 9,  j\neq 16, 17$. 
Thus $m_k\in \{16, 17\}$ or $m_k\leq 8$ by \eqref{1020}. From $2m_k>Q_n$, we further 
obtain $n\in \{10, 20\}$. Let $n=10$. Then considering primitive divisors of $U_5$ and 
$U_{10}$, we have $m_k\geq 10$. This with $m_k\leq 11$ gives $10\leq m_k\leq 11$ 
which is not possible. Thus $n=20$ and further $16\leq m_k\leq 17$. From \eqref{2|n} with 
$p_0=5$ and $|\al^{n/q}-\be^{n/q}|>\frac{2|\al|^{n/q}}{q}$ for each $q|n$, we obtain 
\begin{align*}
\begin{split}
|U_n|=\left|\frac{\al^n-\be^n}{\al^{\frac{n}{5}}-\be^{\frac{n}{5}}}
\frac{\al^{\frac{n}{5}}-\be^{\frac{n}{5}}}{\al-\be}\right|
\geq \frac{7}{8}\times 11|\al|^{\frac{2n}{5}}\frac{2|\al|^{\frac{n}{5}}}{5}\frac{1}{2|\al|}
=\frac{77|\al|^{\frac{3n}{5}-1}}{20}.
\end{split}
\end{align*}
Hence from $\cD=|U_n|, n=20=2^2\cdot 5$ and \eqref{1020}, we obtain 
\begin{align*}
\log |\al|\leq &\frac{1}{11}\left(\frac{\log n-
\frac{2-2^{-1}}{6-2^{-1}}\log (7\cdot 2^{-1}\cdot 11)}{
1-\frac{2-2^{-1}}{6-2^{-1}}-0.6501}-\log \frac{77}{20}\right)<2.234
\end{align*}
or $|\al|<9.333$. For $(r, s)$ given by Lemma \ref{rsBD} with $3\leq |\al|<9.333$ and 
$P(U_{20})\leq P(B_{17})=31$, we check that \eqref{eqn3} does not hold at $n=20$. 

\subsection{$n\in \{16, 24\}$}

We need to consider $n\in \{16, 24, 48\}$. Let $n=48$. Taking 
$n_0=16=2^4, t=0$, we have $\ep=h(16), \lambda=0$ in \eqref{Dub} 
and hence in \eqref{Yev}.  From \eqref{3q|n} with $p^l=2$, we have 
$c_1=3\times 0.56Q_{16}/3=0.56\times 17/3\geq 3.17$ and 
$c_2=\frac{2}{3}$ and hence \eqref{Yev} gives 
\begin{align*}
\frac{\log 48-\frac{\log (24\times 3.17)}{\frac{19}{11}}}{1-\frac{11}{19}-h(16)}
\geq \log Y \begin{cases}
C_{24} & {\rm if } \ Y=\cD\\
47! & {\rm if } \ Y=\cF
\end{cases}
\end{align*}
since \eqref{logal}. This is a contradiction. 

Let $n=16$ and $Y=\cF$. Recall that $|\al|\geq 3$. Taking $n_0=8=2^3, t=1$,  we have  
$\ep=g(8)=\frac{2.2}{4}, \lambda=0$ in \eqref{Dub} and in \eqref{Yev}.  
From \eqref{2|n} with $p_0=2^2$, we have 
$c_1=\frac{7\times 17}{8}$ and $c_2=1$ so that  \eqref{Yev} gives 
\begin{align*}
\frac{\log 16-\frac{3\log \frac{7\cdot 17}{2}}{7}}{1-\frac{3}{7}-0.55}
\geq \cF\geq \log m_k!
\end{align*}
implying $m_k\leq 21$.  By \eqref{Mnub812} in Lemma \ref{MnBDF}, we have 
$M_{8}(\cF)\leq 0.3\cF$. Using $0.3$ in place of $g(8)=0.55$ and proceeding as before, 
we have the above inequality with $0.55$ replaced by $0.3$.  This gives a contradiction 
since  $m_k\geq Q_{16}=17$.  

Finally we consider $n=24$ in the next section. 

\subsection{$n=24$}

Let $n=24$. We write $\log |U_{24}|=23\log |\al|+\log |\frac{1-x^{24}}{1-x}|$ where 
$x=\be/\al$ as before and $|x|=1$. We have
\begin{align*}
\left|\frac{1-x^{24}}{1-x}\right|=\left|\frac{(1-x^{12})(1+x^4)}{1-x}\frac{(1+x^{12})}{1+x^4}\right|
\leq 24\left|\frac{1+x^{12}}{1+x^4}\right|
\end{align*}
by using $|(1-x^l)/(1-x)|\leq l$ for $l\geq 1$ and $|1+x^4|\leq 2$. Hence for any real $\mu>0$, we have 
from Lemma \ref{x3/x} that ,
$$\mu \left|\frac{1-x^{24}}{1-x}\right|-\log \left|\frac{1+x^{12}}{1+x^4}\right|+\log 2
\leq \mu \log 24-(1-\mu)\log 0.56+\log 2<1.273+2.6\mu. 
$$
Taking 
$n_0=24=2^3\cdot 3, t=0$, we have from \eqref{n_0U} in Lemma \ref{primn_0} that 
\begin{align*}
M_{24}(U_{24})&\geq \log \frac{|\al^{12}+\be^{12}|}{|\al^{4}+\be^{4}|}-\log 2
=8\log |\al|+\log \left|\frac{1+x^{12}}{1+x^4}\right|-\log 2\\
&\geq  \frac{8}{23}\left(\log |U_{24}|-\log \left|\frac{1-x^{24}}{1-x}\right|\right)
+\log \left|\frac{1+x^{12}}{1+x^4}\right|-\log 2
\end{align*}
implying 
\begin{align}\label{al-bd}
\log |\al| &\leq \frac{M_{24}(U_{24})-\log \left|\frac{1+x^{12}}{1+x^4}\right|+\log 2}{8}
\leq \frac{M_{24}(U_{24})+\log \frac{2}{0.56}}{8},
\end{align}
by using Lemma \ref{x3/x}, and
\begin{align}\label{U-bd}
\begin{split}
\log |U_{24}|&\leq \frac{23}{8}\left(M_{24}(U_{24})+\frac{8}{23}\log \left|\frac{1-x^{24}}{1-x}\right|
-\log \left|\frac{1+x^{12}}{1+x^4}\right|+\log 2\right)\\
&\leq \frac{23}{8}\left(M_{24}(U_{24})+1.273+\frac{8\times 2.6}{23}\right). 
\end{split}
\end{align}

Let $h>0$ and $\lambda\geq 0$ be such that 
\begin{align}\label{Dub24}
\begin{split}
M_{24}(\cD) &\leq h(\log \cD-\lambda)=h\left(23\log \al+\log \left|\frac{1-x^{24}}{1-x}\right|-\lambda \right).
\end{split}
\end{align}
Since $U_{24}=\cD$, we get from \eqref{al-bd} and \eqref{U-bd} that 
\begin{align}\label{8-23al}
\log |\al| &\leq \frac{h\log \left|\frac{1-x^{24}}{1-x}\right|-\log \left|\frac{1+x^{12}}{1+x^4}\right|+\log 2-h\lambda}{8-23h}<\frac{1.273+h(2.6-\lambda)}{8-23h}
\end{align}
and
\begin{align}\label{8-23U}
\begin{split}
\log |U_{24}|&\leq \frac{23}{8-23h}\left(\frac{8}{23}\log \left|\frac{1-x^{24}}{1-x}\right|
-\log \left|\frac{1+x^{12}}{1+x^4}\right|+\log 2-h\lambda\right)\\
&\leq \frac{23}{8-23h}\left(1.273+\frac{8\times 2.6}{23}-h\lambda\right)
<\frac{50.08-h\lambda}{8-23h}.
\end{split}
\end{align}

By \eqref{Dub}, we have $h=h(24)$ and $\lambda=0$ in \eqref{Dub24}. Hence from 
\eqref{8-23U}, we obtain  
\begin{align*}
\frac{50.08}{8-23h(24)}>\log |U_{24}|\geq 
\begin{cases}
\log m_k! & {\rm if} \ Y=\cF\\
\log C_{m_k} & {\rm if} \ Y=\cD.
\end{cases}
\end{align*}
The above inequality is not valid for $m_k\geq 16,  350$ according as $Y=\cF, \cD,$ respectively. 
Since $m_k\geq Q_{24}=23$, we a contradiction when $Y=\cF$. 
Thus we consider $Y=\cD$ where we have $m_k<350$.  

Let $\delta_0=\frac{1}{2.75}$. First assume that $|1+x^4|<\delta_0$. From $|x|=1$ and the  
inequality $|1-x^4|+|1+x^4|\geq 2$, we obtain 
$$4|1-x|\geq |1-x^4|\geq 2-\delta_0 \quad {\rm implying} \quad \left|\frac{1-x^{24}}{1-x}\right|\leq 
\frac{8}{2-\delta_0}. 
$$
Also 
$$\left|\frac{1+x^{12}}{1+x^4}\right|=|(1+x^4)^2-3x^4|\geq 3-\delta^2_0.$$
Putting this in \eqref{8-23U}, we get 
\begin{align}\label{<del0}
\log C_{m_k}\leq \log \cD=\log |U_{24}|\leq \frac{8\log \frac{8}{2-\delta_0}-23\log \frac{3-\delta^2_0}{2}-h\lambda}{8-23h}.
\end{align}
By \eqref{Dub}, we have $h=f(24)=\frac{2.746}{8}$ and $\lambda=0$ in \eqref{Dub24}. The inequality 
\eqref{<del0} is not valid for $m_k\geq 35$. Thus $m_k<35$. By Lemma \ref{mk<N} with $n_0=24$ and 
$N=34$, we get $M_{24}(\cD)<0.257\log \cD$. Putting $h=0.257, \lambda=0$ in  \eqref{Dub24}, we obtain 
\eqref{<del0} with $h=0.257$. We have $m_k\geq (Q_{24}+1)/2=12$ and \eqref{<del0} is not valid for 
$m_k\geq 12$. This is a contradiction. 

Thus we have $|1+x^4|\geq \delta_0$. We follow the proof as in \cite[Section 4.4]{LLS-C}. Recall that 
$m_k<350$. Suppose $|\al|\leq 50$. Then $\log C_{m_k}\leq \log |U_{24}|\leq 23\log |\al|+\log 24$ gives 
$m_k\leq 72$. We refer to notations of Lemma \ref{mk<N}. Let $t_{1j}+t_{2j}=0$ 
for $j\in \{37, 38, 39, 40, 41, 42, 43\}$. We use Lemma \ref{mk<N} with $n_0=24$ and $N=72$ to 
find that $M_{24}(\cD)<0.29\log \cD$. 
Putting $h=0.29, \lambda=0$ in  \eqref{Dub24}, we obtain \eqref{8-23al} and \eqref{8-23U} with $h=0.29$.  
We obtain $\log |\al|<1.525$ or $|\al|<4.6$ and $m_k\leq 31$ by using $|U_{24}|\geq C_{m_k}$. Again, using 
Lemma \ref{mk<N} with $n_0=24$ and $N=31$, we find that $M_{24}(\cD)\leq 0.257\log \cD$. Putting 
$h=0.257, \lambda=0$ in \eqref{Dub24}, we obtain \eqref{8-23al} and \eqref{8-23U} with $h=0.257$. 
This gives $|\al|\leq 2.54$ and $m_k\leq 21$.  Furthermore we have
$P(U_{24})\leq P(B_{21})\leq 41$.  We check that the equation \eqref{eqn3} with $P(U_{24})\leq 41$ and 
$\al\leq 2.54$ is not possible. Here, we use Lemma \ref{rsBD}  to find all possible pairs $(r, s)$ with 
$\al\leq 2.54$.  Thus, we assume  $t_{1j}+t_{2j}>0$ for some $j\in \{37, 38, 39, 40, 41, 42, 43\}$. From 
$$\log C_{37}+\log C_{39}>23\log 50+\log 24\geq \log |U_{24}|,$$
we obtain $j\leq m_k\leq 43$. Hence, $P(U_{24})\leq P(B_{43})\leq 81$. Further 
$47\cdot 53\cdot 59\cdot 61\cdot 67\cdot 71\cdot 73\mid U_{24}$ 
also since $47\cdot 53\cdot 59\cdot 61\cdot 67\cdot 71\cdot 73\mid C_{j}\mid B_j$ for $j\in \{37, 38, 39, 40, 41, 42, 43\}$. Also 
$$\log |\al|\geq \frac{\log C_{37}-\log 24}{23}>2.0379.$$ 
For the pairs $(r, s)$ with $2.0379<\log |\al| \leq \log 50$ given by Lemma \ref{rsBD}, we check that 
both $P(U_{24})\leq 89$ and $47\cdot 53\cdot 59\cdot 61\cdot 67\cdot 71\cdot 73\mid U_{24}$  are 
not possible. Therefore the  equation \eqref{eqn3} has no solution when $|\al|\leq 50$.

From now on, we assume that $|\al|>50$. Suppose that $t_{1j}=0$ for $j\in \{37, 38\}$. 
We use Lemma \ref{mk<N} with $n_0=24$ and $N=350$ to find that $M_{24}(\cD)<0.324\log \cD$. 
Putting $h=0.324, \lambda=0$ in  \eqref{Dub24}, we obtain \eqref{8-23al} with $h=0.324$. This gives 
$\log |\al|<\log 50$ which is a contradiction.  Therefore, we have $t_{1j}>0$ for $j=37$ or $j=38$. 
By  Lemma \ref{mk<N} with $n_0=24$ and $N=350$, we can take 
$h=\ep_0=0.3433$ and $\lambda=\sum_jt_{1j}\lambda_{1j}+\sum_jt_{2j}\lambda_{2j}$ in \eqref{Dub24} which 
gives 
\begin{align*}
\log |\al|<\frac{1.273+0.3433\left(2.6-\sum_jt_{1j}\lambda_{1j}-\sum_jt_{2j}\lambda_{2j}\right)}{8-23\times 0.3433}
\end{align*}
by \eqref{8-23al}. Since $|\al|>50$, this gives 
\begin{align}\label{24main}
\sum_jt_{1j}\lambda_{1j}+\sum_jt_{2j}\lambda_{2j}<\frac{2.166-0.104\log 50}{0.3433}<5.13.
\end{align}
Recall that 
$$\lambda_{1j}=\left(1-\frac{\ep_{1j}}{\ep_0}\right)\log C_j
\quad {\rm and} \quad \lambda_{2j}=\left(1-\frac{\ep_{2j}}{\ep_0}\right)\log B_{j}.$$
We compute the values of $\lambda_{1j}$ and $\lambda_{2j}$ for $j<350$ and find that 
$$
\lambda_{1j}\leq 5.13 \quad {\rm for} \quad j\in T_1:=\{j: j\leq 6\}\cup \{12, 13, 37, 38, 39, 40\},
$$
and 
$$
\lambda_{2j}\leq 5.13 \quad {\rm for} \quad j\in  T_2:=\{j: j\leq 4\}\cup \{37, 38\}.
$$
Thus, by \eqref{24main},  we may suppose that $t_{1j}>0$ implies $j\in T_1$ and $t_{2j}>0$ implies 
$j\in T_2$.  Recall that we have $t_{1j}>0$ for  $j=37$ or $j=38$.  Write $t_4, t_5, t_{12}$ for $t_{1j}$ 
according as $j=4, 5, 12$, respectively. We find that $\lambda_{1j}\geq 2.639, 3.737, 3.111$ according 
to whether $j=4, 5, 12$, respectively. Hence, from \eqref{24main}, we have 
$t_4\leq 1, t_5\leq 1$ and $t_{12}\leq 1$. Put   
\begin{align*}
\log \cD_1:=\sum_{j\in T_1, j\neq 4, 5, 12}t_{1j}\log C_j+\sum_{j\in T_2}t_{2j}\log B_j,
\end{align*}
so that 
\begin{align}\label{80}
\log \cD&=\log \cD_1+t_4\log C_4+t_5\log C_5+t_{12}\log C_{12}.
\end{align}
We now consider $M_8(\cD)$ given by \eqref{Mn}. 
We find that $M_8(C_j)<0.46\log C_j$ for all $j\in T_1$ except when $j\in \{ 4, 5, 12\}$ and  
$M_8(B_j)<0.46\log B_j$ for $j\in T_2$ and further 
$$M_8(C_4)\leq 0.74, \quad M_8(C_5)\leq 0.53 \quad {\rm and} \quad M_8(C_{12})\leq 0.65. 
$$ 
Hence, from \eqref{80}, $\log |U_{24}|=23\log |\al|+\left|\frac{1-x^{24}}{1-x}\right|$ and 
$\left|\frac{1-x^{24}}{1-x}\right|\leq 12$, we get   
\begin{align*}
M_8(\cD)< &0.46\log \cD_1+0.74t_4\log C_4+0.53t_5\log C_5+0.65t_{12}\log C_{12}\\
<&0.46\log \cD+0.28t_4\log C_4+0.07t_5\log C_5+0.19t_{12}\log C_{12}\\
< &0.46(23\log |\al|+\log 12+\log (1+x^{12}))\\
&+0.28\log C_4+0.07\log C_5+0.19\log C_{12},
\end{align*}
since $t_4, t_5, t_{12}\leq 1$. 
Comparing the above inequality with the lower bound of $M_8(\cD)=M_8(U_{24})$ 
given by \eqref{n_0U} with $n_0=2^3$ and $t=0$, we obtain 
\begin{align*}
4.48>&0.28\log C_4+0.07\log C_5+0.19\log C_{12}+0.46\log 12\\
>&(12-0.46\times 23)\log |\al|+(1-0.46)\log |1+x^{12}|-\log 2\\
>&(12-0.46\times 23)\log 50-\log 2+(1-0.46)\log 0.56>4.54
\end{align*}
since $\al>50$, provided $|1+x^{12}|>0.56$, which is  a contradiction. We show that 
$|1+x^{12}|>0.56$. We have $|1+x^4|=\delta\geq \delta_0=\frac{1}{2.75}$. Hence 
\begin{align*}
|1+x^{12}|=\begin{cases}
|1+x^{4}||(1+x^{12})^2-3x^4|\geq \delta_0|3-1|>0.56 & {\rm if} \ \delta\leq 1\\
|1+x^{4}|\left|\frac{1+x^{12}}{1+x^4}\right|>1\times 0.56=0.56 & {\rm if} \ \delta>1
\end{cases}
\end{align*}
by using Lemma \ref{x3/x}. Therefore the equation \eqref{eqn3} has no solution with $n=24$ 
when $\al$ and $\be$ are complex conjugates.  This completes the proof of Theorem \ref{thmF}. 
\qed

\section*{Acknowledgements}

The author acknowledges the support of SERB MATRICS Project.


\begin{thebibliography}{9999}



\bibitem{BHV} Y. Bilu, G. Hanrot and P. M. Voutier, \emph{Existence of primitive divisors of Lucas and Lehmer numbers, with an appendix by M. Mignotte}, J. Reine Angew. Math.,  {\bf 539} (2001), 75--122.


\bibitem{LLS-F}   S. Laishram, F. Luca and M. Sias, \emph{On members of Lucas sequences which are products of factorials},   Monatsh. Math., {\bf 193} (2020), 329?359.

\bibitem{LLS-C}   S. Laishram, F. Luca and M. Sias, \emph{On members of Lucas sequences which are products of Catalan numbers},  Int. Jour. Number Theory, accepted for publication. 









\end{thebibliography}
\end{document}